\documentclass[11pt]{article}
\usepackage{times}

\usepackage[paperwidth=199.8mm,paperheight=297mm,centering,hmargin=15mm,vmargin=2cm]{geometry}
\usepackage[dvipsnames]{xcolor}
\usepackage{framed}
\definecolor{shadecolor}{rgb}{0.9,0.9,0.9}
\usepackage{mathtools}
\usepackage{amsmath}
\usepackage[shortlabels]{enumitem}
\usepackage[titletoc,title]{appendix}
\usepackage{graphicx,epic,eepic,epsfig,amsmath,latexsym,amssymb,verbatim,color}
\usepackage{dsfont}

\usepackage{float}
\usepackage{tikz}
\usepackage[strict]{changepage}
\usepackage{hyperref}
\hypersetup{colorlinks=true,citecolor=blue,linkcolor=blue,filecolor=blue,urlcolor=blue,breaklinks=true,linktocpage=true}

\usepackage{url}
\usepackage{theorem}
\newtheorem{definition}{Definition}
\newtheorem{proposition}[definition]{Proposition}
\newtheorem{lemma}[definition]{Lemma}

\newtheorem{theorem}[definition]{Theorem}

\usepackage{colortbl}
\usepackage{pifont}
\definecolor{Gray}{gray}{0.92}
\definecolor{Gray2}{gray}{0.75}
\definecolor{maroon}{cmyk}{0,0.87,0.68,0.32}
\usepackage{booktabs} 

\def\squareforqed{\hbox{\rlap{$\sqcap$}$\sqcup$}}
\def\qed{\ifmmode\squareforqed\else{\unskip\nobreak\hfil
\penalty50\hskip1em\null\nobreak\hfil\squareforqed
\parfillskip=0pt\finalhyphendemerits=0\endgraf}\fi}
\def\endenv{\ifmmode\;\else{\unskip\nobreak\hfil
\penalty50\hskip1em\null\nobreak\hfil\;
\parfillskip=0pt\finalhyphendemerits=0\endgraf}\fi}
\newenvironment{proof}{\noindent \textbf{{Proof~} }}{\hfill $\blacksquare$}

\newcounter{remark}
\newenvironment{remark}[1][]{\refstepcounter{remark}\par\medskip\noindent%
\textbf{Remark~\theremark #1} }{\medskip}

\newcounter{example}

\mathchardef\ordinarycolon\mathcode`\:
\mathcode`\:=\string"8000
\def\vcentcolon{\mathrel{\mathop\ordinarycolon}}
\begingroup \catcode`\:=\active
  \lowercase{\endgroup
  \let :\vcentcolon
  }

\usepackage{cleveref}
\usepackage{graphicx}
\usepackage{xcolor}

\RequirePackage[framemethod=default]{mdframed}
\newmdenv[skipabove=7pt,
skipbelow=7pt,
backgroundcolor=darkblue!15,
innerleftmargin=5pt,
innerrightmargin=5pt,
innertopmargin=5pt,
leftmargin=0cm,
rightmargin=0cm,
innerbottommargin=5pt,
linewidth=1pt]{tBox}

\newmdenv[skipabove=7pt,
skipbelow=7pt,
backgroundcolor=blue2!25,
innerleftmargin=5pt,
innerrightmargin=5pt,
innertopmargin=5pt,
leftmargin=0cm,
rightmargin=0cm,
innerbottommargin=5pt,
linewidth=1pt]{dBox}
\newmdenv[skipabove=7pt,
skipbelow=7pt,
backgroundcolor=darkkblue!15,
innerleftmargin=5pt,
innerrightmargin=5pt,
innertopmargin=5pt,
leftmargin=0cm,
rightmargin=0cm,
innerbottommargin=5pt,
linewidth=1pt]{sBox}
\definecolor{darkblue}{RGB}{0,76,156}
\definecolor{darkkblue}{RGB}{0,0,153}
\definecolor{blue2}{RGB}{102,178,255}

\newcommand{\nc}{\newcommand}
\nc{\rnc}{\renewcommand}
\nc{\beg}{\begin{equation}}
\nc{\eeq}{{\end{equation}}}
\nc{\beqa}{\begin{eqnarray}}
\nc{\eeqa}{\end{eqnarray}}
\nc{\lbar}[1]{\overline{#1}}
\nc{\bra}[1]{\langle#1|}
\nc{\ket}[1]{|#1\rangle}
\nc{\ketbra}[2]{|#1\rangle\!\langle#2|}
\nc{\braket}[2]{\langle#1|#2\rangle}

\nc{\proj}[1]{| #1\rangle\!\langle #1 |}
\nc{\avg}[1]{\langle#1\rangle}
\nc{\Rank}{\operatorname{Rank}}
\nc{\smfrac}[2]{\mbox{$\frac{#1}{#2}$}}
\nc{\tr}{\operatorname{Tr}}
\nc{\ox}{\otimes}
\nc{\dg}{\dagger}
\nc{\dn}{\downarrow}
\nc{\cA}{{\cal A}}
\nc{\cB}{{\cal B}}
\nc{\cC}{{\cal C}}
\nc{\cD}{{\cal D}}
\nc{\cE}{{\cal E}}
\nc{\cF}{{\cal F}}
\nc{\cG}{{\cal G}}
\nc{\cH}{{\cal H}}
\nc{\cI}{{\cal I}}
\nc{\cJ}{{\cal J}}
\nc{\cK}{{\cal K}}
\nc{\cL}{{\cal L}}
\nc{\cM}{{\cal M}}
\nc{\cN}{{\cal N}}
\nc{\cO}{{\cal O}}
\nc{\cP}{{\cal P}}
\nc{\cQ}{{\cal Q}}
\nc{\cR}{{\cal R}}
\nc{\cS}{{\cal S}}
\nc{\cT}{{\cal T}}
\nc{\cV}{{\cal V}}
\nc{\cX}{{\cal X}}
\nc{\cY}{{\cal Y}}
\nc{\cZ}{{\cal Z}}
\nc{\cW}{{\cal W}}
\nc{\csupp}{{\operatorname{csupp}}}
\nc{\qsupp}{{\operatorname{qsupp}}}
\nc{\var}{{\operatorname{var}}}
\nc{\rar}{\rightarrow}
\nc{\lrar}{\longrightarrow}
\nc{\polylog}{{\operatorname{polylog}}}
\nc{\1}{I}
\nc{\wt}{{\operatorname{wt}}}
\nc{\av}[1]{{\left\langle {#1} \right\rangle}}
\nc{\supp}{{\operatorname{supp}}}

\def\x{\xi}

\def\o{\omega}

\nc{\RR}{{{\mathbb R}}}
\nc{\CC}{{{\mathbb C}}}
\nc{\FF}{{{\mathbb F}}}
\nc{\NN}{{{\mathbb N}}}
\nc{\ZZ}{{{\mathbb Z}}}
\nc{\PP}{{{\mathbb P}}}
\nc{\QQ}{{{\mathbb Q}}}
\nc{\UU}{{{\mathbb U}}}
\nc{\EE}{{{\mathbb E}}}
\nc{\id}{{\operatorname{id}}}

\nc{\CHSH}{{\operatorname{CHSH}}}

\nc{\be}{\begin{equation}}
\nc{\ee}{{\end{equation}}}
\nc{\bea}{\begin{eqnarray}}
\nc{\eea}{\end{eqnarray}}
\nc{\<}{\langle}
\rnc{\>}{\rangle}
\nc{\rU}{\mbox{U}}

\nc{\ob}[1]{#1}

\nc{\SEP}{{\text{\rm SEP}}}
\nc{\cSEP}{\textbf{SEP}}
\nc{\NS}{{\text{NS}}}
\nc{\LOCC}{{\text{LOCC}}}
\nc{\PPT}{{\text{\rm PPT}}}
\nc{\cPPT}{{\textbf{PPT}}}
\nc{\EXT}{{\text{EXT}}}
\nc{\Sym}{{\operatorname{Sym}}}

\nc{\ERLO}{{E_{\text{r,LO}}}}
\nc{\ERLOCC}{{E_{\text{r,LOCC}}}}
\nc{\ERPPT}{{E_{\text{r,PPT}}}}
\nc{\ERLOCCinfty}{{E^{\infty}_{\text{r,LOCC}}}}
\nc{\Aram}{{\operatorname{\sf A}}}

\usepackage{tikz}
\RequirePackage{doi}
\usepackage{hyperref}
\hypersetup{colorlinks=true,citecolor=blue,linkcolor=blue,filecolor=blue,urlcolor=blue,breaklinks=true}

\makeatletter
\def\grd@save@target#1{%
  \def\grd@target{#1}}
\def\grd@save@start#1{%
  \def\grd@start{#1}}
\tikzset{
  grid with coordinates/.style={
    to path={%
      \pgfextra{%
        \edef\grd@@target{(\tikztotarget)}%
        \tikz@scan@one@point\grd@save@target\grd@@target\relax
        \edef\grd@@start{(\tikztostart)}%
        \tikz@scan@one@point\grd@save@start\grd@@start\relax
        \draw[minor help lines,magenta] (\tikztostart) grid (\tikztotarget);
        \draw[major help lines] (\tikztostart) grid (\tikztotarget);
        \grd@start
        \pgfmathsetmacro{\grd@xa}{\the\pgf@x/1cm}
        \pgfmathsetmacro{\grd@ya}{\the\pgf@y/1cm}
        \grd@target
        \pgfmathsetmacro{\grd@xb}{\the\pgf@x/1cm}
        \pgfmathsetmacro{\grd@yb}{\the\pgf@y/1cm}
        \pgfmathsetmacro{\grd@xc}{\grd@xa + \pgfkeysvalueof{/tikz/grid with coordinates/major step}}
        \pgfmathsetmacro{\grd@yc}{\grd@ya + \pgfkeysvalueof{/tikz/grid with coordinates/major step}}
        \foreach \x in {\grd@xa,\grd@xc,...,\grd@xb}
        \node[anchor=north] at (\x,\grd@ya) {\pgfmathprintnumber{\x}};
        \foreach \y in {\grd@ya,\grd@yc,...,\grd@yb}
        \node[anchor=east] at (\grd@xa,\y) {\pgfmathprintnumber{\y}};
      }
    }
  },
  minor help lines/.style={
    help lines,
    step=\pgfkeysvalueof{/tikz/grid with coordinates/minor step}
  },
  major help lines/.style={
    help lines,
    line width=\pgfkeysvalueof{/tikz/grid with coordinates/major line width},
    step=\pgfkeysvalueof{/tikz/grid with coordinates/major step}
  },
  grid with coordinates/.cd,
  minor step/.initial=.2,
  major step/.initial=1,
  major line width/.initial=2pt,
}
\makeatother
\usepackage{tcolorbox}
\usepackage{relsize}
\usepackage{graphicx}
\usepackage{caption}
\usepackage{subcaption}
\usetikzlibrary{shapes}

\usetikzlibrary{plotmarks}

\nc{\st}{\text{subject to} \ }
\nc{\supre}{\text{supremum} \ }
\nc{\sdp}{\text{sdp}}

\newcommand{\herm}{\text{\rm Herm}}

\newcommand{\sfT}{\mathsf{T}}

\newcommand{\Sep}{\text{\rm Sep}} 
\newcommand{\Sepcone}{\mathcal{S}\mathcal{E}\mathcal{P}} 
\renewcommand{\SEP}{\Sepcone}
\newcommand{\DPS}{\text{\rm DPS}}
\newcommand{\DPScone}{\mathcal{D}\mathcal{P}\mathcal{S}}
\newcommand{\PPTcone}{\mathcal{P}\mathcal{P}\mathcal{T}}
\newcommand{\EXTcone}{\mathcal{E}\mathcal{X}\mathcal{T}}

\newcommand{\Herm}{\text{Herm}}
\DeclareMathOperator{\conv}{conv}

\newcommand{\TT}{\mathcal{T}}
\newcommand{\bh}{\bar{h}}
\newcommand{\B}{\mathsf{B}}
\newcommand{\C}{\mathsf{C}}
\newcommand{\Tr}{\mathsf{T}}
\newcommand{\psd}{\geq}

\newcommand{\gd}{\tilde{F}} 

\newcommand{\PM}{\mathrm{P}}

\newcommand{\rr}{\rho}

\begin{document}

\title{\Large \textbf{The sum-of-squares hierarchy on the sphere,\\
and applications in quantum information theory}}

\date{}

\author{\normalsize Kun Fang~\thanks{Department of Applied Mathematics and Theoretical Physics, University of Cambridge, UK. \ \emph{kf383@cam.ac.uk}} \and \normalsize Hamza Fawzi~\thanks{Department of Applied Mathematics and Theoretical Physics, University of Cambridge, UK. \ \emph{h.fawzi@damtp.cam.ac.uk}}}

\maketitle

\begin{abstract}

We consider the problem of maximizing a homogeneous polynomial on the unit sphere and its hierarchy of Sum-of-Squares (SOS) relaxations. Exploiting the \emph{polynomial kernel technique}, we obtain a quadratic improvement of the known convergence rate by Reznick and Doherty~\&~Wehner. Specifically, we show that the rate of convergence is no worse than $O(d^2/\ell^2)$ in the regime $\ell \geq \Omega(d)$ where $\ell$ is the level of the hierarchy and $d$ the dimension, solving a problem left open in the recent paper by de Klerk \& Laurent (arXiv:1904.08828). Importantly, our analysis also works for matrix-valued polynomials on the sphere which has applications in quantum information for the Best Separable State problem.  By exploiting the duality relation between sums of squares and the DPS hierarchy in quantum information theory, we show that our result generalizes to nonquadratic polynomials the convergence rates of Navascu{\'e}s, Owari \& Plenio.

\end{abstract}

\tableofcontents

\newpage

\section{Introduction}

We consider in this paper a fundamental computational task, that of maximizing a multivariate polynomial $p \in \RR[x]$ in $d$ variables $x=(x_1,\ldots,x_d)$ on the unit sphere:
\begin{equation}
\label{eq:maxp-intro}
p_{\max} = \max_{x \in S^{d-1}} p(x)
\end{equation}
where $S^{d-1} = \{x \in \RR^d : x_1^2 + \dots + x_d^2 = 1\}$. Optimization problems of the above form have applications in many areas. For example, computing the largest stable set of a graph is a special case of \eqref{eq:maxp-intro} for a suitable polynomial $p$ of degree three, see \cite{nesterov2003random,deklerksurvey}. Computing the $2\rightarrow 4$ norm of a matrix $A$ corresponds to the maximization of the degree-four polynomial $p(x) = \|Ax\|_4^4$ on the sphere, see e.g., \cite{barak2012hypercontractivity} for more on this. In quantum information, the so-called \emph{Best Separable State problem} very naturally relates to polynomial optimization on the sphere, as we explain later. 

When $p(x)$ is quadratic, problem \eqref{eq:maxp-intro} reduces to an eigenvalue problem which can be solved efficiently. However for general polynomials of degree greater than two, the problem  is NP-hard as it contains as a special case the stable set problem \cite{nesterov2003random}.  The sum-of-squares hierarchy is a hierarchy of semidefinite relaxations that approximate the value $p_{\max}$ by a sequence of semidefinite programs of increasing size \cite{parrilo2000structured,lasserre2001global}. In this paper we study the approximation quality of this sequence of semidefinite relaxations.

\subsection{Sum-of-squares hierarchy}

 The sum-of-squares hierarchy to approximate \eqref{eq:maxp-intro} is defined by
\[
p_{\ell} = \min \left\{ \gamma \in \RR \text{ s.t. } \gamma - p \text{ is sum-of-squares of degree $\ell$ on } S^{d-1} \right\}.
\]
The sequence $(p_{\ell})_{\ell \in \NN}$ consists of monotone upper bounds on $p_{\max}$, i.e., for any $\ell$ we have $p_{\max} \leq p_{\ell}$ and $p_{\ell} \leq p_{\ell-1}$. For each $\ell$, the value $p_{\ell}$ can be computed by a semidefinite program of size $d^{O(\ell)}$, see e.g.,~\cite{parrilo2000structured,lasserre2001global}.

A result of Reznick \cite{Reznick1995} (see also \cite{Doherty2012}) shows that $p_{\ell} \rightarrow p_{\max}$ as $\ell \rightarrow \infty$. In fact Reznick shows, assuming $p_{\min} = \min_{x \in S^{d-1}} p(x) = 0$, that $p_{\ell}/p_{\max}$ converges to 1 at the rate $d/\ell$, for $\ell$ large enough. In this paper we show that the sum-of-squares hierarchy actually converges at the faster rate of $(d/\ell)^2$. More precisely, we prove the following

\begin{theorem}
\label{thm:main}
Assume $p(x_1,\ldots,x_d)$ is a homogeneous polynomial of degree $2n$ in $d$ variables with $n \leq d$, and let $p_{\min}$ denote the minimum of $p$ on $S^{d-1}$. Then for any $\ell \geq \C_n d$
\begin{equation}
\label{eq:pellmain}
1 \leq \frac{p_{\ell} - p_{\min}}{p_{\max} - p_{\min}} \leq 1 + (\C_n d/\ell)^2
\end{equation}
for some constant $\C_n$ that depends only on $n$.
\end{theorem}

In a recent paper, de Klerk and Laurent \cite{klerklaurentsphere} proved that a semidefinite hierarchy of \emph{lower bounds} on $p_{\max}$ converges at a rate of $O(1/\ell^2)$ and left open the question of whether the same is true for the hierarchy $(p_{\ell})$ of upper bounds. Our Theorem \ref{thm:main} answers this question positively.

\subsection{Matrix-valued polynomials}

The proof technique we use in this paper actually allows us to get a significant generalization of Theorem \ref{thm:main}, related to matrix-valued polynomials. Let $\mathbf{S}^k$ be the space of real symmetric matrices of size $k\times k$, and let $\mathbf{S}^k[x]$ be the space of $\mathbf{S}^k$-valued polynomials in $x=(x_1,\ldots,x_d)$. We will often use the lighter notation $F \in \mathbf{S}[x]$ when the size $k$ is unimportant for the discussion. A polynomial $F(x_1,\ldots,x_d) \in \mathbf{S}[x]$ is \emph{positive} if $F(x) \geq 0$ for all $x \in \RR^d$ where the inequality is interpreted in the positive semidefinite sense. We say that $F(x) \in \mathbf{S}^k[x]$ is a \emph{sum of squares} if there exist polynomials $U_j(x) \in \RR^{k\times k}[x]$ such that $F(x) = \sum_{j} U_j(x) U_j(x)^{\sfT}$ for all $x \in \RR^d$. We say that $F(x)$ is \emph{$\ell$-sos on $S^{d-1}$} if it agrees with a sum-of-squares polynomial on the sphere with $\deg U_j \leq \ell$. We are now ready to state our main theorem on sum of squares representations for matrix-valued polynomials.

\begin{theorem}
\label{thm:main1}
Assume $F(x_1,\ldots,x_d) \in \mathbf{S}[x]$ is a homogeneous matrix-valued polynomial of degree $2n$ in $d$ variables with $n \leq d$, such that $F(x)$ is symmetric for all $x$. Assume furthermore that $0 \leq F(x) \leq \1$ for all $x \in S^{d-1}$, where $I$ is the identity matrix. There are constants $\C_n$ and $\C'_n$ that depend only on $n$ such that for any $\ell \geq \C_n d$, $F + \C'_n\left(\frac{d}{\ell}\right)^2 I$ is $\ell$-sos on $S^{d-1}$.
\end{theorem}

Some remarks concerning the statement are in order:
\begin{itemize}
\item Theorem \ref{thm:main} is a direct corollary of Theorem \ref{thm:main1} where $F(x)$ is the scalar polynomial given by $F(x) = (p_{\max} - p) / (p_{\max} - p_{\min})$.
\item A remarkable fact of Theorem \ref{thm:main1} is that the result is totally independent on the size of the matrix $F(x)$.
\item Theorem \ref{thm:main1} can be applied to get sum-of-squares certificates for scalar \emph{bihomogeneous} polynomials on products of two spheres $S^{k-1} \times S^{d-1}$. Indeed, one way to think about a matrix-valued polynomial $F(x_1,\ldots,x_d) \in \mathbf{S}^k[x]$ is to consider the real-valued polynomial $p(x,y) = y^\sfT F (x) y$ where $x \in \RR^d$ and $y \in \RR^k$. This polynomial is bihomogeneous of degree $(2n,2)$ in the variables $(x,y)$. One important application of this setting is in quantum information theory for the best separable state problem which we explain later in the paper.
\item As stated, Theorem \ref{thm:main1} is concerned only with levels $\ell \geq \Omega(d)$ of the sum-of-squares hierarchy. The main technical result we prove in this paper (Theorem \ref{thm:maintxt} below) actually allows us to get a bound on the performance of the sum-of-squares hierarchy for \emph{all} values of level $\ell$, and not just the regime $\ell \geq \Omega(d)$. The bounds we get however do not have closed-form expressions in general, and they depend on the eigenvalues of some generalized Toeplitz matrices. For small values of $n$ (namely $2n=2$ and $2n=4$) our bounds can be computed efficiently though, as we explain later.
\item For more details about the regime $\ell = o( d )$ of the sum-of-squares hierarchy, we refer the reader to the recent works \cite{bhattiprolu,Barak2017} and references therein. 
\end{itemize}

\subsection{The Best Separable State problem in quantum information theory}

The notion of entanglement plays a fundamental role in quantum mechanics. The set of separable states (i.e., non-entangled states) on the Hilbert space $\CC^d \otimes \CC^d$ is defined as the convex hull of all pure product states
\begin{equation}
\label{eq:Sepintro}
\Sep(d) = \conv \left\{ xx^{\dagger} \otimes yy^{\dagger} : (x, y) \in \CC^d \times \CC^d \text{ and } \|x\| = \|y\| = 1\right\}.
\end{equation}
Here $x^{\dagger} = \bar{x}^{\sfT}$ is the conjugate transpose and $\|x\|^2 = x^{\dagger} x = \sum_{i=1}^d |x_i|^2$. $\Sep(d)$ is a convex subset of the set $\Herm(d^2)$ of Hermitian matrices of size $d^2 \times d^2$. A key computational task in quantum information theory is the so-called \emph{Best Separable State} (BSS) problem: given $M \in \Herm(d^2)$, compute
\begin{equation}
\label{eq:hsep}
h_{\Sep}(M) \; = \; \max_{\rho \in \Sep(d)} \tr[M\rho] \; = \; \max_{\substack{x,y \in \CC^d\\ \|x\| = \|y\| = 1}} \sum_{1\leq i,j,k,l \leq d} M_{ij,kl} x_i \bar{x}_k y_j \bar{y}_l.
\end{equation}
In words, $h_{\Sep}(M)$ is the \emph{support function} of the convex set $\Sep(d)$ evaluated at $M$. Note that $h_{\Sep}(M)$ is simply the maximum of the \emph{Hermitian polynomial}\footnote{A Hermitian polynomial is a polynomial of complex variables and their conjugates that takes only real values. See Section \ref{sec:hermitianpoly} for more details.}
\begin{equation}
\label{eq:pMintro}
p_M(x,\bar x,y,\bar y) := \sum_{1\leq i,j,k,l \leq d} M_{ij,kl} x_i \bar{x}_k y_j \bar{y}_l
\end{equation}
over the product of spheres $S_{\CC^d} \times S_{\CC^d} = \{(x,y) \in \CC^d \times \CC^d : \|x\| = \|y\| = 1\}$. In that sense the BSS problem is very related to the polynomial optimization problem \eqref{eq:maxp-intro}.

The Doherty-Parrilo-Spedalieri (DPS) hierarchy \cite{Doherty} is a hierarchy of semidefinite relaxations to the set of separable states, which is defined in terms of so-called \emph{state extensions} (we recall the precise definitions later in the paper). It satisfies
\[
    \Sep(d) \subseteq \cdots \subseteq \DPS_{\ell}(d) \subseteq \cdots \subseteq \DPS_2(d) \subseteq \DPS_1(d)
\]
where $\DPS_{\ell}(d)$ is the $\ell$'th level of the DPS hierarchy. It turns out that the DPS hierarchy can be interpreted, from the dual point of view, as a sum of squares hierarchy. This duality relation has been mentioned multiple times in the literature, however we could not find any formal and complete proof of this equivalence. In this paper we give a proof of this duality relation. To do this, we first need to specify the definition of \emph{sum of squares} for Hermitian polynomials. We say that a Hermitian polynomial is a \emph{real sum of squares (rsos)} if it can be written as a sum of squares of Hermitian polynomials.\footnote{Another common definition is to require that the polynomial is a sum of squares of modulus squares of (holomorphic) complex polynomials. This is a different condition, and it corresponds from the dual point of view to the DPS hierarchy without the Positive Partial Transpose conditions. See Section \ref{sec:hermitianpoly} for more details on this.}  To state the result it is more convenient to work in the conic setting and we denote the convex cones associated to $\Sep$ and $\DPS_k$ by $\Sepcone$ and $\DPScone_k$ respectively (these convex cones simply correspond to dropping a trace normalization condition).
\begin{theorem}[Duality DPS/sum-of-squares]
Let $\Sepcone(d)$ be the convex cone of separable states on $\CC^d \otimes \CC^d$, and let $\DPScone_{\ell}(d)$ be the convex cone of quantum states corresponding to the $\ell$'th level of the DPS hierarchy. Then we have:\\
(i) $\Sepcone(d)^* = \left\{ M \in \herm(d^2) : p_M \text{ is nonnegative} \right\}$\\
(ii) $\DPScone_{\ell}(d)^* = \left\{ M \in \herm(d^2) : \|y\|^{2(\ell-1)} p_M \text{ is a real sum-of-squares} \right\}$,\\
where $K^*$ denotes the dual cone to $K$ and $p_M$ is the Hermitian polynomial of Equation \eqref{eq:pMintro}.
\end{theorem}
Using this connection, our results on the convergence of the sum-of-squares hierarchy  can be easily translated to bound the convergence rate of the DPS hierarchy. More precisely, since the polynomial $p_M$ of Equation \eqref{eq:pMintro} is bihomogeneous of degree $(2,2)$ (i.e., it is quadratic in $x$ and $y$ independently) we can get a bound on the rate of convergence of the DPS hierarchy from Theorem \ref{thm:main1} where $\deg F = 2$. The rate of convergence we get in this way actually coincides with the rate of convergence obtained by Navascues, Owari and Plenio \cite{Navascues2009}, who use a completely different (quantum-motivated) argument based on the primal definition of the DPS hierarchy using state extensions. From the sum-of-squares point of view, the theorem of Navascues et al. can thus be seen as a special case of Theorem \ref{thm:main1} when $\deg F = 2$. We conclude by stating the result on the convergence rate of the DPS hierarchy.

\begin{theorem}[Convergence rate of DPS hierarchy, see also \cite{Navascues2009}]
\label{thm:convergencedps}
Let $M \in \herm(d^2)$ and assume that $(x\ox y)^{\dagger} M (x\ox y) \geq 0$ for all $(x,y) \in \CC^d \times \CC^d$. Then
\[
h_{\Sep}(M) \leq h_{\DPS_{\ell}}(M) \leq (1 + \C d^2 / \ell^2) h_{\Sep}(M)
\]
for any $\ell \geq \C' d$, where $\C,\C' > 0$ are absolute constants.
\end{theorem}

\subsection{Overview of proof}
\label{sec:overview}

We give a brief overview of the proof of Theorem \ref{thm:main1}. We will focus on the case where $F(x)$ is a scalar-valued polynomial for simplicity of exposition.

Given a univariate polynomial $q(t)$ of degree $\ell$ consider the \emph{kernel} $K(x,y) = q(\langle x , y \rangle)^2$ for $(x,y) \in S^{d-1} \times S^{d-1}$. Define the integral transform, for $h : S^{d-1} \rightarrow \RR$
\begin{equation}
\label{eq:Kgtransform}
(Kh)(x) = \int_{y \in S^{d-1}} K(x,y) h(y) d\sigma(y) \qquad \forall x \in S^{d-1}
\end{equation}
where $d\sigma$ is the rotation-invariant probability measure on $S^{d-1}$. If $h \geq 0$ then the function $Kh$ is $\ell$-sos\footnote{Because of the integral, $Kh$ is an ``infinite'' sum of squares. Standard convexity results can be used however to turn this into a finite sum of squares.} on $S^{d-1}$, by construction of the kernel $K(x,y)$.

Let $F(x)$ be a scalar-valued polynomial such that $0 \leq F(x) \leq 1$ on $S^{d-1}$. Our goal is to find $\delta > 0$ such that $\gd = F + \delta$ is $\ell$-sos. Assuming that the mapping $K$ is invertible, we can always write $\gd = Kh$ with $h = K^{-1} \gd$. If $K$ is close to the identity (i.e., the kernel $K(x,y)$ is close to a Dirac kernel $\delta(x,y)$) then we expect that $h \approx \gd$, i.e., that $\|h - \gd\|_{\infty}$ is small. Since $\gd \geq \delta$, if we can guarantee that $\|h - \gd\|_{\infty} \leq \delta$ it would follow that $h \geq 0$, in which case the equation $\gd = Kh = K(K^{-1} \gd)$ gives a degree-$\ell$ sum-of-squares representation of $\gd$.

To make the argument above precise we need to measure how close the kernel $K$ is to the identity. This is best done in the Fourier domain, where we analyze how close the Fourier coefficients of the the kernel $K(x,y)$ are to 1. The Fourier coefficients of $K(x,y)$ depend in a quadratic way on the coefficients in the expansion of $q(t)$ in the basis of Gegenbauer polynomials. We show that there is a choice of $q(t)$ such that the Fourier coefficients of $K(x,y)$ converge to 1 at the rate $\frac{d^2}{\ell^2}$, as $\ell \rightarrow \infty$. The kernel we construct is the solution of an eigenvalue maximization for a generalized Toeplitz matrix, associated to the family of Gegenbauer polynomials. We use known results on the roots of such polynomials to obtain the desired rate of convergence.

The idea of proof here is similar to the approaches in Reznick \cite{Reznick1995}, and Doherty \& Wehner \cite{Doherty2012}, and Parrilo \cite{Parrilo-unpublished}. The work of Reznick uses the kernel $K(x,y) = \langle x , y \rangle^{2\ell} / c$ for some normalizing constant $c$ for which the Fourier coefficients can be computed explicitly.\footnote{The fact that Reznick's proof is based on this choice of kernel was observed by Blekherman in \cite[Remark 7.3]{blekherman2004convexity}.} The Fourier coefficients of this kernel happen to converge to 1 at a rate of $\frac{d}{\ell}$, which is slower than the kernels we construct.

\subsection*{Organization}

In Section \ref{sec:bg} we review some background material concerning Fourier decompositions on the sphere. The proof of Theorem \ref{thm:main1} is in Section \ref{sec:proofmain}. Section \ref{sec:qi} is devoted to the Best Separable State problem in quantum information theory.

\section{Background}
\label{sec:bg}

\paragraph{Spherical harmonics} We review the basics of Fourier analysis on the sphere $S^{d-1}$. Any polynomial $p$ of degree $n$ on the sphere has a unique decomposition
\begin{equation}
\label{eq:fouriersphere}
p = p_0 + p_1 + \dots + p_n, \qquad p_i \in \cH^d_i
\end{equation}
where each $p_i$ is a \emph{spherical harmonic of degree $i$}. The decomposition \eqref{eq:fouriersphere} is known as the Fourier-Laplace decomposition of $p$. The space $\cH^d_i$ is defined as the restriction on $S^{d-1}$ of the set of homogeneous harmonic polynomials of degree $i$, i.e.,
\[
\cH^d_i = \left\{ f|_{S^{d-1}} : f \in \RR[x_1,\ldots,x_d], \text{ homogeneous of degree $i$ and } \Delta f = \sum_{k=1}^{d} \frac{\partial^2 f}{\partial x_k^2} = 0 \right\}.
\]
Equivalently, the spaces $\cH^d_i$ are also the irreducible subspaces of $L^2(S^{d-1})$ under the action of $SO(d)$. For example $\cH_0^d$ is the set of constant functions, $\cH_1^d$ is the set of linear functions, and $\cH_2^d$ is the set of traceless quadratic forms. The spaces $\cH^d_i$ are mutually orthogonal with respect to the $L^2$ inner product $\langle f , g \rangle = \int fg d\sigma$ where $d\sigma$ is the rotation-invariant probability measure on the sphere. Note that if $p$ is an \emph{even} polynomial (i.e., $p(x) = p(-x)$) then the only nonzero harmonic components of $p$ are the ones of even order.

\paragraph{Integral transforms} Consider a general $SO(n)$-invariant kernel $K(x,y) = \phi(\langle x , y \rangle)$ where $\phi$ is some univariate polynomial of degree $L$. The kernel $K$ acts on functions $f:S^{d-1} \rightarrow \RR$ as follows
\[
(Kf)(x) = \int_{S^{d-1}} K(x,y) f(y) d\sigma(y).
\]
To understand the action of $K$ on arbitrary polynomials $f$, it is very convenient to decompose $\phi$ into the basis of \emph{Gegenbauer polynomials} (also known as \emph{ultraspherical polynomials}) $(C_k(t))_{k \in \NN}$ which are orthogonal polynomials on $[-1,1]$ with respect to the weight $(1-t^2)^{\frac{d-3}{2}} dt$. Using appropriate normalization (which we adopt here) these polynomials satisfy the following important property:
\[
\int_{S^{d-1}} C_k(\langle x , y \rangle) p_i(y) d\sigma(y) = \delta_{ik} p_i(x) \qquad \forall x \in S^{d-1}
\]
for any $p_i \in \cH^d_i$. In other words, the kernel $(x,y) \mapsto C_k(\langle x , y \rangle)$ is a \emph{reproducing kernel} for $\cH_k^d$. Now going back to the kernel $K(x,y) = \phi(\langle x , y \rangle)$, if we expand $\phi = \lambda_0 C_0 + \lambda_1 C_1 + \dots + \lambda_L C_L$, then it follows that for any polynomial $p$ with Fourier expansion \eqref{eq:fouriersphere} we have
\begin{equation}
\label{eq:funkhecke}
(Kp)(x) = \int_{y \in S^{d-1}} K(x,y) p(y) d\sigma(y) = \lambda_0 p_0(x) + \lambda_1 p_1(x) + \dots + \lambda_L p_L(x).
\end{equation}
The equation above tells us that the harmonic decomposition $\cH_0 \oplus \cH_1 \oplus \dots$ diagonalizes $K$, with the Gegenbauer coefficients $(\lambda_i)_{i=0,\ldots,L}$ being the eigenvalues. Equation \eqref{eq:funkhecke} is also known as the \emph{Funk-Hecke formula}. The coefficients $(\lambda_i)_{i=1,\ldots,L}$ in the expansion of $\phi$ in the basis of Gegenbauer polynomials are given by the following integral
\begin{equation}
\label{eq:ggcoeffs}
\lambda_i = \frac{\omega_{d-1}}{\omega_{d}} \int_{-1}^{1} \phi(t) \frac{C_i(t)}{C_i(1)} (1-t^2)^{\frac{d-3}{2}} dt
\end{equation}
where $\omega_d$ is the surface area of $S^{d-1}$. If we let $w(t) = (1-t^2)^{\frac{d-3}{2}}$, one can check that $\int_{-1}^{1} C_i(t)^2 w(t) dt = \frac{\o_d}{\o_{d-1}} C_i(1)$; in other words, $\sqrt{\frac{\o_{d-1}}{\o_{d}}} \frac{C_i(t)}{\sqrt{C_i(1)}}$ has unit norm with respect to $w(t) dt$.

\begin{remark}
Note that if the univariate polynomial $\phi(t)$ is nonnegative on $[-1,1]$, then the coefficients $\lambda_0,\ldots,\lambda_L$ in \eqref{eq:ggcoeffs} satisfy $\lambda_i \leq \lambda_0$ for all $i=0,\ldots,L$ since $C_i(t) \leq C_i(1)$ for all $t \in [-1,1]$. We will use this simple property of the coefficients later in the proof.
\end{remark}

\paragraph{A technical lemma} The following lemma will be important for our proof later. It shows that the sup-norm of the harmonic components of a polynomial $f$ can be bounded by a constant independent of the dimension $d$, times the sup-norm of $f$.

\begin{proposition}\label{prop:boundinfproj}
  For any integer $n$ there exists a constant $\B_{2n}$ such that the following is true. For any homogeneous polynomial $f$ with degree $2n$ and with decomposition into spherical harmonics $f = \sum_{k=0}^{n} f_{2k}$ with $f_{j} \in \cH_j^d$ it holds $\|f_{2k}\|_{\infty} \leq \B_{2n} \|f\|_{\infty}$. Also $\B_2 \leq 2$ and $\B_4 \leq 9$.
\end{proposition}
\begin{proof}
The proof is in Appendix \ref{proof 1}.
\end{proof}

\medskip

\noindent The remarkable property in the previous proposition is that the constant $\B_{2n}$ is independent of the dimension $d$.

\begin{remark}
\label{rem:shiftf}
When $f$ is a homogeneous polynomial of degree $2n$ such that $0 \leq m \leq f \leq M$ on $S^{d-1}$, Proposition \ref{prop:boundinfproj} gives us that $\|f_{2k}\|_{\infty} \leq \B_{2n} M$. However one can get a better bound by applying Proposition \ref{prop:boundinfproj} instead to $f - (m+M)/2$; this gives $\|f_{2k}\|_{\infty} \leq (M-m)/2$ for all $k=1,\ldots,n$.
\end{remark}

\section{Proof of Theorem \ref{thm:main1}}
\label{sec:proofmain}

In this section we prove our main theorem, Theorem \ref{thm:main1}. We will actually prove a more general result giving bounds on the performance of the sum-of-squares hierarchy for \emph{all} values of the level $\ell$. (In Theorem \ref{thm:main1} stated in the introduction, only the regime $\ell \geq \Omega(d)$ was presented.)

For the statement of our theorem we need to introduce two quantities that play an important role in our analysis.
\begin{itemize}
\item The first quantity, which we denote $\rr_{2n}(d,\ell)$, is defined as (where $n,d,\ell$ are integers)
\begin{equation}
\label{eq:rrdef}
\rr_{2n}(d,\ell) = \min_{\substack{q \in \RR[t], \deg(q)=\ell\\ \lambda_0=1}} \quad \sum_{k=1}^{2n} |\lambda_{2k}^{-1} - 1|.
\end{equation}
Here, the minimization is over polynomials $q(t)$ of degree $\ell$, and $\lambda_{2k}$ is the $2k$'th coefficient of $\phi(t) = (q(t))^2$ in its Gegenbauer expansion, see Equation \eqref{eq:ggcoeffs}. In words, $\rr_{2n}(d,\ell)$ quantifies how close we can get the Gegenbauer coefficients of $\phi(t) = (q(t))^2$ to 1 (note however that the distance to 1 is measured by $|\lambda_{2k}^{-1} - 1|$ and not linearly).
\item The second quantity is the constant $\B_{2n}$ introduced in Proposition \ref{prop:boundinfproj}. It is the smallest constant such that for any homogeneous polynomial $f$ of degree $2n$, we have $\|f_{2k}\|_{\infty} \leq \B_{2n} \|f\|_{\infty}$ for all $k=0,\ldots,n$, where $f_{2k}$ are the $2k$'th harmonic components of $f$. In other words, $\B_{2n}$ is an upper bound on the $\infty\rightarrow \infty$ operator norm of the linear map that projects a homogeneous polynomial of degree $2n$ onto its $2k$'th harmonic component. Proposition \ref{prop:boundinfproj} says that such an upper bound that only depends on $n$ (i.e., independent of $d$) does exist. One can get explicit upper bounds on $\B_{2n}$ for small values of $n$. For example one can show that $\B_{2} \leq 2$ and $\B_{4} \leq 9$.
\end{itemize}

We are now ready to state our main theorem:

\begin{theorem}
\label{thm:maintxt}
Assume $F(x_1,\ldots,x_d)$ is a homogeneous matrix-valued polynomial of degree $2n$ in $d$ variables, such that $F(x)$ is symmetric for all $x$, and $0 \leq F \leq \1$ on $S^{d-1}$. Then $F + (\B_{2n}/2) \rho_{2n}(d,\ell) \1$ is $\ell$-sos on $S^{d-1}$.

Furthermore, the quantity $\rho_{2n}(d,\ell)$ satisfies the following: for any $n \leq d$, there are constants $\C_n,\C'_n$ such that for $\ell \geq \C'_n d$, $\rho_{2n}(d,\ell) \leq \C_n (d/\ell)^2$.
\end{theorem}

\begin{proof}[Proof of first part of Theorem \ref{thm:maintxt}]
We will start by proving the first part of the theorem. For clarity of exposition, we will assume that $F$ is a scalar-valued polynomial, and we explain later why the argument also works for matrices. Let thus $F$ be a homogeneous polynomial of degree $2n$ such that $0 \leq F \leq 1$ on $S^{d-1}$. Let
\[
F = F_0 + F_2 + \dots + F_{2n} \qquad (F_{2k} \in \cH^d_{2k})
\]
be the decomposition of $F$ into spherical harmonics (since $F$ is even, only harmonics of even order are nonzero). Given $\delta > 0$ to be specified later, we will exhibit a sum-of-squares decomposition of $\gd = F+\delta$ by writing $\gd =K K^{-1} \gd$ where $K$ is an integral transform defined as
\begin{align}\label{eq: Kf definition}
    (Kh)(x):= \int_{S^{d-1}} \phi(\<x,y\>) h(y) d\sigma(y), \quad \forall x \in S^{d-1}
\end{align}
where $\phi(t) = (q(t))^2$ is a univariate polynomial of degree $2\ell$. In order for $\gd = KK^{-1} \gd$ to be a valid sum-of-squares decomposition of $\gd$, we need that $K^{-1} \gd \geq 0$. The polynomial $q(t)$ will be chosen so that $K$ is close to a Dirac kernel; when combined with $\gd \geq \delta > 0$ we will be able to conclude that $K^{-1} \gd \geq 0$ from the fact that $\|\gd - K^{-1} (\gd)\|_{\infty} \leq \delta$.

Let $(\lambda_{i})_{0 \leq i \leq 2\ell}$ be the coefficients in the Gegenbauer expansion of $\phi$, i.e., $\phi = \lambda_0 C_0 + \lambda_1 C_1 + \dots + \lambda_{2\ell} C_{2\ell}$. 
 By the Funk-Hecke formula we have $K^{-1} (\gd) = \lambda_0^{-1} (F_0+\delta) + \lambda_2^{-1} F_2 + \dots + \lambda_{2n}^{-1} F_{2n}$.  Our analysis does not depend on the scaling of $K$ so we will assume $\lambda_0 = 1$.
Thus we get
\[
\|K^{-1} (\gd) - \gd\|_{\infty} = \left\|\sum_{k=1}^{n} \left(\frac{1}{\lambda_{2k}} - 1\right) F_{2k}\right\|_{\infty} \leq \sum_{k=1}^{n} \left|\frac{1}{\lambda_{2k}} - 1\right| \|F_{2k}\|_\infty \leq (\B_{2n}/2) \sum_{k=1}^{n} \left|\frac{1}{\lambda_{2k}} - 1\right|
\]
where in the last inequality we used Proposition \ref{prop:boundinfproj} (see also Remark \ref{rem:shiftf}) together with the fact that $0 \leq F \leq 1$. It thus follows that if 
\begin{equation}
\label{eq:condlambda}
(\B_{2n}/2) \sum_{k=1}^n |\lambda_{2k}^{-1} - 1| \leq \delta
\end{equation}
then $K^{-1} (\gd) \geq 0$ and the equation $\gd = KK^{-1}(\gd)$ gives a valid sum-of-squares decomposition of $\gd=F+\delta$. We have thus proved the first part of Theorem \ref{thm:main}.
\end{proof}

It now remains to prove the second part of the theorem, which leads us to the analysis of the quantity $\rr_{2n}(d,\ell)$. Before doing so, we explain how the proof above applies in the case where $F$ is a matrix-valued polynomial.

\paragraph{Matrix-valued polynomials} Assume $F \in \mathbf{S}[x]$ homogeneous of degree $2n$. We can decompose each entry of $F$ into spherical harmonics to get $F = F_0 + F_2 + \dots + F_{2n}$. Define $\gd = F(x) + \delta \1$ for a $\delta > 0$ to be specified later. The steps in the argument above are identical, where $\|\cdot\|_{\infty}$ is defined as the maximum of $\|F(x)\|$ over $x \in S^{d-1}$, where $\|F(x)\|$ is the spectral norm of $F(x)$, and the bound on $\|F_{2k}\|_{\infty}$ follows from Proposition \ref{prop:boundinfprojmatrix}. If $\|K^{-1} \gd - \gd \|_{\infty} \leq \delta$ then $K^{-1}\gd \geq 0$ in the positive semidefinite sense. Letting $H = K^{-1} \gd \geq 0$, we get $\gd(x) = (K H)(x) = \int_{S^{d-1}} q(\langle x, y \rangle)^2 H(y) d\sigma(y) = \int_{S^{d-1}} U_y(x) U_y(x)^{\sfT} d\sigma(y)$ where $U_y(x) = q(\langle x , y \rangle) H(y)^{1/2}$ is a polynomial of degree $\ell$ in $x$. This is what we wanted.

\medskip

We now proceed to the analysis of $\rr_{2n}(d,\ell)$.

\paragraph{Reformulating $\rr_{2n}(d,\ell)$ using generalized Toeplitz matrices} It will be convenient to reformulate the optimization problem \eqref{eq:rrdef} in terms of certain suitable (generalized) Toeplitz matrices. We parametrize the degree-$\ell$ polynomial $q(t)$  as
\[
q(t) = \sum_{i=0}^{\ell} e_i \frac{C_i(t)}{\sqrt{C_i(1)}}
\]
where $e_0,\ldots,e_{\ell} \in \RR$. The presence of the term $\sqrt{C_i(1)}$ is for convenience later. The Gegenbauer coefficients of $\phi(t) = (q(t))^2$ are then equal to (cf. Equation \eqref{eq:ggcoeffs})
\[
\begin{aligned}
\lambda_k &= \frac{\omega_{d-1}}{\omega_{d}} \int_{-1}^{1} \phi(t) \frac{C_k(t)}{C_k(1)} (1-t^2)^{\frac{d-3}{2}} dt\\
&=\sum_{i,j=0}^{\ell} e_i e_j \left( \frac{\o_{d-1}}{\o_{d}} \int_{-1}^1  \frac{C_i(t)}{\sqrt{C_i(1)}} \frac{C_j(t)}{\sqrt{C_j(1)}}\frac{C_{k}(t)}{C_{k}(1)}(1-t^2)^{\frac{d-3}{2}} dt\right)\\
&= e^{\sfT} \TT\left[C_k / C_k(1)\right] e,
\end{aligned}
\]
where for $h:[-1,1]\rightarrow \RR$, $\TT[h]$ is the $(\ell+1) \times (\ell+1)$ symmetric matrix
\[
\TT[h]_{i,j} = \frac{\o_{d-1}}{\o_{d}} \int_{-1}^1  \frac{C_i(t)}{\sqrt{C_i(1)}} \frac{C_j(t)}{\sqrt{C_j(1)}} h(t) (1-t^2)^{\frac{d-3}{2}} dt.
\]
It can be easily checked that $\TT[1] = \1$ is the identity matrix (this follows from the fact that the polynomials $\sqrt{\frac{\omega_{d-1}}{\omega_d}} \frac{C_i}{\sqrt{C_i(1)}}$ have unit norm with respect to the weight function $(1-t^2)^{(d-3)/2}$), and so $\lambda_0 = e^{\sfT} e = \sum_{k} e_k^2$. It thus follows that $\rho_{2n}(d,\ell)$ can be formulated as:
\begin{equation}
\label{eq:rrmatrixform}
\rho_{2n}(d,\ell) = \min_{\substack{e \in \RR^{\ell+1}\\ \sum_{k} e_k^2 = 1}} \;\; \sum_{k=1}^{n} \left|\left(e^{\sfT} \TT\left[\frac{C_{2k}}{C_{2k}(1)}\right] e\right)^{-1} - 1\right|.
\end{equation}

\paragraph{Case $2n=2$} Let us first analyze the case $2n=2$ which corresponds to quadratic polynomials. In this case the sum in \eqref{eq:rrmatrixform} has simply one term. It is then not difficult to see that $\rho_{2}(d,\ell)$ is given by
\begin{equation}
\label{eq:rho2simplified}
\rho_{2}(d,\ell) = \left\|\TT\left[\frac{C_{2}}{C_{2}(1)}\right]\right\|^{-1} - 1
\end{equation}
where $\|\cdot\|$ denotes the spectral norm. Thus we see that $\rr_{2}(d,\ell)$ can be computed efficiently by simply evaluating the spectral norm of $\TT[C_2 / C_2(1)]$. The latter matrix can be formed explicitly using known formulas for the integrals of Gegenbauer polynomials (see e.g., \cite{Hsu1938}). Note that $\TT[C_2 / C_2(1)]$ is a banded matrix with bandwidth 3.

\paragraph{Case $2n=4$} We now turn to quartic polynomials. In this case $\rr_4(d,\ell)$ takes the form
\begin{equation}
\label{eq:rrquartic}
\rho_{4}(d,\ell) = \min_{\substack{e \in \RR^{\ell+1}\\ \sum_{k} e_k^2 = 1}} \;\; \left|\left(e^{\sfT} \TT\left[\frac{C_{2}}{C_{2}(1)}\right] e\right)^{-1} - 1\right| \; + \; \left|\left(e^{\sfT} \TT\left[\frac{C_{4}}{C_{4}(1)}\right] e\right)^{-1} - 1\right|.
\end{equation}
Let $\mathcal{R}$ be the \emph{joint numerical range} (also known as the field of values) of the matrices $\TT\left[\frac{C_{2}}{C_{2}(1)}\right]$ and $\TT\left[\frac{C_{4}}{C_{4}(1)}\right]$, i.e.,
\[
\mathcal{R} = \left\{ \left(e^{\sfT} \TT\left[\frac{C_{2}}{C_{2}(1)}\right] e \;\; , \;\; e^{\sfT} \TT\left[\frac{C_{4}}{C_{4}(1)}\right] e \right) : e \in \RR^{\ell+1}, \sum_{k=0}^{\ell} e_k^2 = 1 \right\}.
\]
From results about joint numerical ranges, it is known that $\mathcal{R} \subset \RR^2$ is convex, see \cite{brickman1961field} and also \cite[Theorem 5.6]{polik2007survey}. It is not difficult to see then that $\mathcal{R}$ has a semidefinite representation, and that $\rr_{4}(d,\ell)$ can be computed using semidefinite programming.

\paragraph{General degree $2n$} We now analyze the case of general degree $2n$. To do this we formulate a proxy for the optimization problem that defines $\rr_{2n}(d,\ell)$ that is easier to analyze. Instead of minimizing $\sum_{k=1}^{2n} |\lambda_{2k}^{-1} - 1|$ we will seek instead to minimize $\sum_{k=1}^{2n} (1-\lambda_{2k})$. Since $\lambda_{2k} \leq \lambda_0 = 1$, both problems seek to bring the $\lambda_{2k}$ close to 1, but the latter problem is easier to analyze because it is linear in the $\lambda_{2k}$. Define
\begin{equation}
\label{eq:defrhotilde}
\tilde{\rr}_{2n}(d,\ell) = \min_{\substack{e \in \RR^{\ell+1}\\ \sum_{i} e_i^2=1}} \sum_{k=1}^{n} \bigl(1-e^{\sfT} \TT[C_{2k}/C_{2k}(1)] e\bigr).
\end{equation}
Since $\TT$ is linear, i.e., $\TT[h_1 + h_2] = \TT[h_1] + \TT[h_2]$ we get that $\tilde{\rr}_{2n}(d,\ell) = n - n\lambda_{\max}(\TT[h])$ where $h = \frac{1}{n} \sum_{k=1}^n C_{2k} / C_{2k}(1)$. It thus remains to analyze $\lambda_{\max}(\TT[h])$. This is what we do next.

\begin{proposition}
\label{prop:lambdamaxh}
Let $h = \frac{1}{n} \sum_{k=1}^{n} \frac{C_{2k}}{C_{2k}(1)}$.  Then $\lambda_{\max}(\TT[h]) \geq 1 - \frac{7n}{12} \frac{d^2}{\ell^2}$.
\end{proposition}
\begin{proof}
We use the following standard result on orthogonal polynomials which gives the eigenvalues of $\TT[f]$ for any linear polynomial $f$. (The result below is stated in full generality for clarity, in our case the $(p_k)$ is the family of normalized Gegenbauer polynomials.)
\begin{proposition}[Standard result on orthogonal polynomials]
\label{prop:eigTlinear}
Let $(p_{k})_{k \in \NN}$ be a family of orthogonal polynomials with respect to a weight function $w(x) > 0$. We assume the $(p_k)$ are normalized, i.e., $\int p_k^2 w = 1$.
Given a linear polynomial $f$, define the $(\ell+1)\times (\ell+1)$ matrix
\begin{equation}
\label{eq:matrixT}
\TT[f]_{ij} = \int_{a}^{b}  p_i(t) p_j(t) f(t) w(t) dt \qquad \forall 0 \leq i,j \leq \ell.
\end{equation}
Then the eigenvalues of $\TT[f]$ are precisely the $f(x_{\ell+1,i})$ where the $(x_{\ell+1,i})_{i=1,\ldots,\ell+1}$ are the roots of $p_{\ell+1}$.
\end{proposition}
\begin{proof}
This follows from standard results on orthogonal polynomials. When $f = 1$ then $\TT[f]$ is the identity matrix. When $f(t) = t$, the matrix $\TT[f]$ is the tridiagonal matrix that encodes the three-term recurrence formula for the $(p_k)$. It is well-known that the eigenvalues of this tridiagonal matrix are the roots of $p_{\ell+1}$. See e.g., \cite[Lemma 3.9]{Parter1965}.
\end{proof}

\medskip

Our function $h(t) = \frac{1}{n} \sum_{k=1}^n \frac{C_{2k}(t)}{C_{2k}(1)}$ is not linear. However one can verify (see Proposition \ref{prop:Cilb}) that it is lower bounded by its linear approximation at $t=1$, i.e., we have
\[
h(t) \geq h'(1)(t-1) + h(1).
\]

It is easy to check that if $h_1,h_2$ are two functions such that $h_1(t) \geq h_2(t)$ for all $t \in [-1,1]$, then $\TT[h_1] \geq \TT[h_2]$ (positive semidefinite order) and thus the largest eigenvalue of $\TT[h_1]$ is at least the largest eigenvalue of $\TT[h_2]$. Let $\bh(t) = h'(1)(t-1) + h(1)$. The largest eigenvalue of $\TT[\bh]$ is equal to $\bh(x_{\ell+1,\ell+1})$ where $x_{\ell+1,\ell+1}$ is the largest root of $C_{\ell+1}$. It is known \cite[Section 2.3 (last displayed equation)]{driver2012bounds} that $x_{\ell+1,\ell+1}$ satisfies
\[
x_{\ell+1,\ell+1} \geq 1 - \frac{1}{4} \frac{d^2}{\ell^2}.
\]
It thus follows, using the fact that $h(1) = 1$ and $h'(1) > 0$, that
\[
\lambda_{\max}(\TT[h]) \geq \lambda_{\max}(\TT[\bh]) = \bh(x_{\ell+1,\ell+1}) \geq -h'(1) \frac{d^2}{4\ell^2} + 1 \geq 1 - \frac{7n}{12} \cdot \frac{d^2}{\ell^2},
\]
where in the last inequality we used the exact value of $h'(1)$ given by $h'(1) = (n+1)(3d+4n-4)/(3(d-1))$ and the fact that $n \leq d$.
\end{proof}

Our proof of Theorem \ref{thm:maintxt} (i) is now almost complete. We just need to relate $\tilde{\rr}$ back to $\rr$. We use the following easy proposition.
\begin{proposition}
If $\tilde{\rr} < 1$ then $\rho \leq \tilde{\rr} / (1-\tilde{\rr})$.
\end{proposition}
\begin{proof}
Let $(\lambda_{2k})$ be the optimal choice in the solution to $\tilde{\rr}$ (Equation \eqref{eq:defrhotilde}). Then $\lambda_{2k} = 1 - (1-\lambda_{2k}) \geq 1-\tilde{\rr} > 0$. Thus $\sum_{k=1}^{n} |\lambda_{2k}^{-1} - 1| = \sum_{k=1}^{n} (1-\lambda_{2k})/\lambda_{2k} \leq \tilde{\rr}/(1-\tilde{\rr})$.
\end{proof}

Proposition \ref{prop:lambdamaxh} tells us that $\tilde{\rr}_{2n}(d,\ell) \leq (7n^2/12) (d/\ell)^2$. For $\ell \geq 2nd$, we will have $\tilde{\rr}_{2n}(d,\ell) \leq 1/2$ and so $\rr_{2n}(d,\ell) \leq 2 \tilde{\rr}_{2n}(d,\ell) \leq 2n^2 (d/\ell)^2$. This completes the proof of Theorem \ref{thm:maintxt}.

\paragraph{Tightness} Our analysis of $\rho_{2n}(d,\ell)$ in the regime $\ell \geq \Omega(d)$ can be shown to be tight. We show this in the case $2n=2$ below.

\begin{theorem}[Tightness of convergence rate]
There is an absolute constant $\C > 0$ such that for $\ell \geq \Omega(d)$, $\rho_2(d,\ell) \geq \C (d/\ell)^2$.
\end{theorem}
\begin{proof}
Given the expression for $\rho_{2}(d,\ell)$ in \eqref{eq:rho2simplified}, we need to produce an upper bound on $\|\TT[C_2 / C_2(1)]\|$. Note that $C_2(t) / C_2(1) = \frac{d}{d-1} t^2 - \frac{1}{d-1}$. It thus follows that $\TT[C_2 / C_2(1)] = \frac{d}{d-1} \TT[t^2] - \frac{\1}{d-1}$.
We now use the following property of generalized Toeplitz matrices constructed from sequences of orthogonal polynomials: If $\TT_{\infty}$ denotes the semi-infinite version of \eqref{eq:matrixT}, then $\TT_{\infty}[f] \TT_{\infty}[g] = \TT_{\infty}[fg]$ for any polynomials $f,g$ (this property follows immediately from the fact that the sequence of orthogonal polynomials $(p_k)_{k=0}^{\infty}$ is an orthonormal basis of the space of polynomials, see e.g., \cite[Lemma 2.4]{baxley1971extreme}). In particular we have $\TT_{\infty}[t^2] = \TT_{\infty}[t]^2$. Now noting that $\TT_{\infty}[t]$ is tridiagonal, we see that $\TT[t^2]$ is a submatrix of $(\TT_{\ell+2}[t])^2$, where the subscript indicates the truncation level (so $\TT_{\ell+2}[t]$ is $(\ell+3)\times (\ell+3)$). Thus it follows that $\TT[C_2 / C_2(1)]$ is a submatrix of $\frac{d}{d-1} (\TT_{\ell+2}[t])^2 - \frac{1}{d-1} \1$. Since $(\TT_{\ell+2}[t])^2$ is positive semidefinite it then follows that
\[
\left\|\TT[C_2 / C_2(1)]\right\| \leq \frac{d}{d-1} \lambda_{\max}(\TT_{\ell+2}[t])^2 - \frac{1}{d-1}
\]
Recall that $\lambda_{\max}(\TT_{\ell+2}[t])$ is the largest root of $C_{\ell+3}$. From \cite[Corollary 2.3]{area2004zeros} we get, for $\ell \geq \Omega(d)$, $\lambda_{\max}(\TT_{\ell+2}[t])^2 \leq 1 - \C (d/\ell)^2$ for some constant $\C$. Thus we get $\rho_2(d,\ell) = \|\TT[C_2 / C_2(1)]\|^{-1} - 1 \geq \C (d/\ell)^2$ as desired.
\end{proof}

\section{Relation to quantum state extendibility}
\label{sec:qi}

Quantum entanglement is one of the key ingredients in quantum information processing. Certifying whether a given state is entangled or not is a hard computational task \cite{gurvits2003classical} and considerable effort has been dedicated to this problem, e.g.,~\cite{Lewenstein2000,Horodecki2001}. Of particular interest is the hierarchy of tests known as the \emph{DPS hierarchy}~\cite{Doherty2002,Doherty}, applying semidefinite programs to verify quantum entanglement.

In this section, we explore the duality relation between the DPS hierarchy and sums of squares, and explain how our results from the previous section can be used to bound the convergence rate of the DPS hierarchy. We show that the result of Navascues et al.~\cite{Navascues2009} can be seen as the special case of our Theorem \ref{thm:maintxt} when the polynomial $F$ is quadratic.

\subsection{Quantum extendible states}
\label{Quantum extendible states and SOS operators}

A quantum state is usually represented by a positive semidefinite operator normalized with unit trace. In this work, we mainly work with \emph{unnormalized quantum states} and consider its convex cone. Given Hilbert spaces $\cH_A \simeq \CC^{d_A}$ and $\cH_B \simeq \CC^{d_B}$, denote the cone of bipartite quantum states as $\cS(\cH_A\ox \cH_B)$, i.e., the cone of positive semidefinite matrices of size $d_A d_B$. A bipartite quantum state $\rho_{AB} \in \cS(\cH_A\ox\cH_B)$ is \emph{separable} if and only if it can be written as a conic combination of tensor product states, i.e.,
\begin{align}\label{sep state}
	\rho_{AB} = \sum_{i} p_i (x_i x_i^{\dagger}) \ox (y_i y_i^{\dagger}) \quad \text{with} \quad p_i \geq 0, x_i \in \cH_A, y_i \in \cH_B.
\end{align}
The convex cone of quantum separable states is denoted as $\SEP(\cH_A \otimes \cH_B)$ and it is strictly included in $\cS(\cH_A \otimes \cH_B)$.

\paragraph{Positive partial transpose} A well-known necessary condition for a state $\rho_{AB}$ to be in $\SEP$ is that it has a \emph{positive partial transpose} (PPT). If we let $\Tr$ denote the transpose operation on Hermitian matrices of size $d_B \times d_B$, then for $\rho_{AB}$ of the form \eqref{sep state} we have
\[
(I \otimes \Tr)(\rho_{AB}) = \sum_{i} p_i (x_i x_i^{\dagger}) \otimes (y_i y_i^{\dagger})^{\sfT} = \sum_{i} p_i (x_i x_i^{\dagger}) \otimes (\bar{y_i} \bar{y_i}^{\dagger}) \psd 0.
\]
If we let $\PPTcone(\cH_A \otimes \cH_B)$ be the set of states with a positive partial transpose then we have the inclusions 
\[
\Sepcone(\cH_A \otimes \cH_B) \subset \PPTcone(\cH_A \otimes \cH_B) \subset \cS(\cH_A \otimes \cH_B).
\]
A well-known result due to Woronowicz \cite{woronowicz1976positive} asserts we have equality $\SEP(\cH_A \otimes \cH_B) = \PPTcone(\cH_A \otimes \cH_B)$ if and only if dimensions of $\cH_A$ and $\cH_B$ satisfy $d_A + d_B \leq 5$.

\paragraph{Extendibility} When the inclusion $\Sepcone \neq \PPTcone$ is strict, one can find more accurate relaxations of $\SEP$ based on the notion of \emph{state extendibility}.  For simplicity of the following discussion, we introduce the notation $[s_1:s_2]:=\{s_1,s_1+1,\cdots,s_2\}$ and $[s]:=[1:s]$ for short. Given a separable state expressed as Eq.~\eqref{sep state} with\footnote{We can always impose such condition without losing generality by changing the coefficients $p_i$ accordingly.} $\|x_i\|=\|y_i\|=1$ we can consider its \emph{extension} (on the $B$ subsystem) as:
\begin{equation}
\label{eq:stateextension}
    \rho_{AB_{[\ell]}} = \sum_{i} p_{i} x_i x_i^{\dagger} \ox \left(y_i y_i^{\dagger} \right)^{\ox \ell}.
\end{equation}
The new system $\rho_{AB_{[\ell]}}$ lies in $\cS(\cH_A \otimes \cH_{B_1} \otimes \dots \otimes \cH_{B_\ell})$ where each $\cH_{B_i} \simeq \CC^{d_B}$; i.e., it is a Hermitian matrix of size $d_A (d_B)^\ell \times d_A (d_B)^\ell$. The system $\rho_{AB_{[\ell]}}$ satisfies a number of  properties, as follows:
\begin{itemize}
\item[(a)] \emph{Positivity}: $\rho_{AB_{[\ell]}}$ is positive semidefinite
\item[(b)] \emph{Reduction under partial traces}: If we \emph{trace out}\footnote{The partial trace operator is the unique linear map $\tr_B : \herm(\cH_A \otimes \cH_B) \rightarrow \herm(\cH_A)$ such that $\tr_B( \rho_A \otimes \sigma_B ) = \tr(\sigma_B) \rho_A$ for all $\rho_A \in \herm(\cH_A)$ and $\sigma_B \in \herm(\cH_B)$.} the systems $B_2,\ldots,B_\ell$ from $\rho_{AB_{[\ell]}}$ we get back the original system $\rho_{AB}$. Indeed we have:
\begin{equation}
\label{eq:extensionpartialtracecond}
\tr_{B_{[2:\ell]}} \rho_{AB_{[\ell]}} = \sum_{i} p_{i} x_i x_i^{\dagger} \ox y_i y_i^{\dagger} \cdot \tr \left[(y_i y_i^{\dagger})^{\ox \ell-1}\right] \overset{(*)}{=} \sum_{i} p_i x_i x_i^{\dagger} \ox y_i y_i^{\dagger} = \rho_{AB}.
\end{equation}
In $(*)$ we used the fact that $\|y_i\| = 1$.
\item[(c)] \emph{Symmetry}: define the \emph{symmetric subspace} of $\cH^{\otimes \ell}$ as
\newcommand{\fS}{\mathfrak{S}}
\[
\Sym(\cH^{\otimes \ell}) = \left\{ Y \in \cH^{\otimes \ell} : P \cdot Y = Y \quad \forall P \in \fS_\ell \right\}
\]
where $\fS_\ell$ is the symmetric group on $\ell$ elements which naturally acts on $\cH^{\otimes \ell}$ by permutation of the components. The dimension of $\Sym(\cH^{\otimes \ell})$ is equal to $\binom{\ell+d-1}{\ell}$ where $d=\dim \cH$. If we let $\Pi = \Pi^{\dagger}$ be the projector on the symmetric subspace of $\cH_B^{\otimes \ell}$ then one can easily verify that $\Pi \left(y y^{\dagger} \right)^{\otimes \ell} \Pi = \left(y y^{\dagger} \right)^{\otimes \ell}$. It thus follows that the extension $\rho_{AB_{[\ell]}}$ of Equation \eqref{eq:stateextension} satisfies
\begin{equation}
\label{eq:extensionsymmetrycond}
(I \otimes \Pi) \rho_{AB_{[\ell]}} (I\otimes \Pi) = \rho_{AB_{[\ell]}}.
\end{equation}
\item[(d)] \emph{Positive Partial Transpose}: If we let $\Tr$ be the transpose map on Hermitian matrices of size $d_B \times d_B$ then $\rho_{AB_{[\ell]}}$ satisfies
\begin{equation}
\label{eq:extensionpptcond}
(I_A \otimes \underbrace{\Tr_{B_1} \otimes \dots \Tr_{B_s}}_{s} \otimes I_{B_{i+1}} \otimes \dots \otimes I_{B_\ell})(\rho_{AB_{[\ell]}}) \psd 0
\end{equation}
for any $s=1,\ldots,\ell$. For convenience later the state on the left of \eqref{eq:extensionpptcond} will be denoted $\rho_{AB_{[\ell]}}^{\sfT_{B_{[s]}}}$.
\end{itemize}

\paragraph{The DPS hierarchy} Define now the set $\DPScone_\ell(\cH_A \otimes \cH_B)$ as
\begin{equation}
\begin{aligned}
\DPScone_\ell(\cH_A \otimes \cH_B) = \Bigl\{ \rho \in \cS(\cH_A \otimes \cH_B) & \text{ s.t. } \exists \rho_{AB_{[\ell]}} \in \cS(\cH_A \otimes \cH_{B_1} \otimes \dots \cH_{B_\ell})\\
& \text{ s.t. conditions } \eqref{eq:extensionpartialtracecond}, \eqref{eq:extensionsymmetrycond}, \eqref{eq:extensionpptcond} \text{ are satisfied} \Bigr\}.
\end{aligned}
\end{equation}
By the previous reasoning, each set $\DPScone_\ell(\cH_A \otimes \cH_B)$ is a convex cone containing $\Sepcone(\cH_A \otimes \cH_B)$, i.e., we have
\[
    \Sepcone \subseteq \cdots \subseteq \DPScone_\ell \subseteq \cdots \subseteq \DPScone_2 \subseteq \DPScone_1 \subseteq \cS.
\]
Note that $\DPScone_1 = \PPTcone$. Also it is known that the hierarchy is complete in the sense that if $\rho \notin \Sepcone$ then there exists a $\ell \in \NN$ such that $\rho \notin \DPScone_\ell$~\cite{Doherty2002,Doherty}.

\begin{remark}
[\;\;(Extendibility without PPT conditions)]
One can also consider the weaker hierarchy where the Positive Partial Transpose constraints are dropped:
\begin{equation}
\label{eq:EXTk}
\begin{aligned}
\EXTcone_\ell(\cH_A \otimes \cH_B) = \Bigl\{ \rho \in \cS(\cH_A \otimes \cH_B) & \text{ s.t. } \exists \rho_{AB_{[\ell]}} \in \cS(\cH_A \otimes \cH_{B_1} \otimes \dots \cH_{B_\ell})\\
& \text{ s.t. conditions } \eqref{eq:extensionpartialtracecond} \text{ and } \eqref{eq:extensionsymmetrycond} \text{ are satisfied} \Bigr\}.
\end{aligned}
\end{equation}
It turns out that this weaker hierarchy $\EXTcone_\ell$ is already complete in the sense stated above. This is usually proven using de Finetti theorems \cite{definettioneandhalf,KMdefinetti}.
\end{remark}

\subsection{Hermitian polynomials and sums of squares}
\label{sec:hermitianpoly}

In this section we leave the quantum world and introduce some terminology pertaining to Hermitian polynomials. A Hermitian polynomial $p(z,\bar{z})$ is a polynomial with complex coefficients in the variables $z=(z_1,\ldots,z_n)$ and $\bar{z} = (\bar{z}_1,\ldots,\bar{z}_n)$ such that $p(z,\bar{z}) \in \RR$ for all $z \in \CC^n$. The general form of a Hermitian polynomial is
\[
p(z,\bar{z}) = \sum_{(u,v) \in A} p_{uv} z^{u} \bar{z}^{v} \qquad (p_{uv} \in \CC)
\]
where the coefficients $p_{uv}$ satisfy $p_{uv} = \overline{p_{vu}}$. We say that $p(z)$ is nonnegative if $p(z) \geq 0$ for all $z \in \CC^n$.
\begin{definition}[Hermitian polynomials and sums of squares]
Let $p(z,\bar{z})$ be a nonnegative Hermitian polynomial.
We say that $p(z,\bar{z})$ is a \emph{real sum-of-squares  (rsos)} if we can write $p(z,\bar{z}) = \sum_{i} g_i(z,\bar{z})^2$ where $g_i(z,\bar{z})$ are Hermitian polynomials. We say that $p(z,\bar{z})$ is a \emph{complex sum-of-squares (csos)} if we can write $p(z,\bar{z}) = \sum_{i} |q_i(z)|^2$ where $q_i(z)$ are (holomorphic) polynomial maps in $z$ (i.e., $q_i$ are functions of $z$ alone and not $\bar{z}$).
\end{definition}
Clearly if $p(z,\bar{z})$ is csos then it is also rsos since $|q(z)|^2 = \text{Re}[q(z)]^2 + \text{Im}[q(z)]^2$ and $\text{Re}[q(z)]$ and $\text{Im}[q(z)]$ are both Hermitian polynomials. The converse however is not true. For example $p(z,\bar{z}) = (z+\bar{z})^2$ is evidently rsos, however it is not csos \cite[Example (a)]{putinardangelo}. Indeed the zero-set of a csos polynomial must be a complex variety, and the zero set of $p(z,\bar{z})$ here is the imaginary axis. Note that a Hermitian polynomial $p(z,\bar{z})$ is rsos iff the \emph{real} polynomial $P(a,b) = p(a+ib,a-ib)$ is a sum-of-squares (in the usual sense for real polynomials).

\subsection{The duality relation}

An element $M \in \herm(d_A d_B)$ is in the dual of $\Sepcone(\cH_A \otimes \cH_B)$ if, and only if the Hermitian polynomial $p_M$ defined by
\begin{align}\label{eq: hermitian polynomial}
p_M(x,\bar{x},y,\bar{y}) \; := \; \sum_{ijkl} M_{ij,kl} \, x_i \,\bar{x}_k \, y_j \, \bar{y}_l, \quad \forall x \in \CC^{d_A}, y \in \CC^{d_B}
\end{align}
is nonnegative for all $(x,y) \in \CC^{d_A} \times \CC^{d_B}$. We prove our first main result on the duality between the DPS hierarchy and sums of squares.
\begin{theorem}[Duality between extendibility hierarchy and sums of squares]
\label{thm:duality}
For $M \in \herm(d_A d_B)$, we let $p_M$ be the associated Hermitian polynomial in \eqref{eq: hermitian polynomial}. Then we have:\\
i) $\Sepcone^* = \left\{M\in \herm(d_A d_B) : p_M \text{ is nonnegative}\right\}$.\\
ii) $\DPScone_\ell^* = \left\{M \in \herm(d_A d_B): \|y\|^{2(\ell-1)} p_M \text{ is rsos}\right\}$.\\
iii) $\EXTcone_\ell^* = \left\{M\in \herm(d_A d_B) : \|y\|^{2(\ell-1)} p_M \text{ is csos}\right\}$.
\end{theorem}
\begin{proof}
Point (i) is immediate and follows from the definition of duality. Points (ii) and (iii) are proved in Appendix \ref{Duality relation}.
\end{proof}

\medskip

The following diagram summarizes the situation.

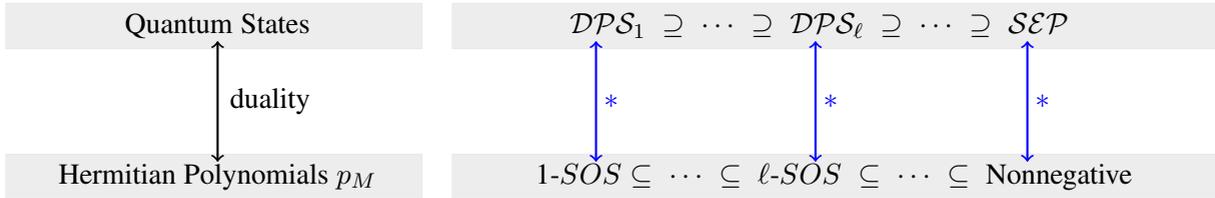
\begin{figure}[H]
    \centering
    \begin{tikzpicture}
  \draw[black!10,fill=black!10,opacity=0.7] (-5.1,0.3) rectangle (5,-0.3);
  \draw[black!10,fill=black!10,opacity=0.7] (-5.1,-1.7) rectangle (5,-2.3);

\draw[black!10,fill=black!10,opacity=0.7] (-11,0.3) rectangle (-5.5,-0.3);
  \draw[black!10,fill=black!10,opacity=0.7] (-11,-1.7) rectangle (-5.5,-2.3);

  \draw[<->, thick] (-8.2,-0.2) node[shift={(0.7,-0.8)}]{duality} --(-8.2,-1.8);


  \node[] at (-8.2,0) {Quantum States};
  \node[] at (-8.2,-2) {Hermitian Polynomials $p_M$};
  \node[] at (-0.25,0) {$\DPScone_1 \; \supseteq \; \cdots \; \supseteq \; \DPScone_\ell \; \supseteq \; \cdots \; \supseteq \; \SEP$};
  \node[] at (-0.05,-2) {$\text{1-}SOS  \subseteq \; \cdots \; \subseteq \; \text{$\ell$-}SOS  \; \subseteq \; \cdots \; \subseteq  \; \text{Nonnegative}$};
  \draw[<->,thick,blue] (-3.2,-0.2) node[shift={(0.2,-0.8)}]{$*$} --(-3.2,-1.8);
  \draw[<->,thick,blue] (-0.3,-0.2) node[shift={(0.2,-0.8)}]{$*$} --(-0.3,-1.8);
  \draw[<->,thick,blue] (2.5,-0.2) node[shift={(0.2,-0.8)}]{$*$} --(2.5,-1.8);
  \end{tikzpicture}
  \caption{A summary of the duality relations between the DPS hierarchy and sums of squares. The notation $\ell$-SOS is a shorthand for $\|y\|^{2(\ell-1)} p_M$ is a real sum-of-squares.}
  \label{duality graph}
\end{figure}

\paragraph{In terms of the support functions} The support function of the set $\DPS_{\ell}$ is defined as
\[
h_{\DPS_{\ell}}(M) = \max_{\rho \in \DPS_{\ell}} \quad \tr[M\rho]
\]
where $M \in \herm(d_A d_B)$.
The duality relation of Theorem \ref{thm:duality} allows us to express $h_{\DPS_{\ell}}(M)$ in the following way:
\[
h_{\DPS_{\ell}}(M) = \min \; \gamma \; \text{ s.t. } \; \|y\|^{2(\ell-1)} ( \gamma \|x\|^2 \|y\|^2 - p_M ) \text{ is  rsos}.
\]
A somewhat more convenient formulation using matrix polynomials is as follows. For $x \in \CC^d$, we let $\tilde{x} \in \RR^{2d}$ be the vector of real and imaginary components of $x$. Given $M \in \herm(d_A d_B)$, let also $\tilde{\PM}_M(\tilde{y}) \in \mathbf{S}[\tilde{y}]$ such that, for any $(x,y) \in \CC^{d_A} \times \CC^{d_B}$ we have
\[
(x\otimes y)^{\dagger} M (x\otimes y) = \tilde{x}^T \tilde{\PM}_M(\tilde{y}) \tilde{x}.
\]
Then one can show the following equivalent formulation of $h_{\DPS_{\ell}}(M)$:
\begin{equation}
\label{eq:hDPSell2}
h_{\DPS_{\ell}}(M) = \min \; \gamma \; \text{ s.t. } \; \gamma I - \tilde{\PM}_M(\tilde{y}) \text{ is $\ell$-sos on $S^{2d_B-1}$}.
\end{equation}
This can be proved using the following lemma, which is a straightforward generalization of \cite{de2005equivalence} to the matrix case.
\begin{lemma}
Let $G(y_1,\ldots,y_d)$ be a homogeneous matrix-valued polynomial of even degree $2n$, such that $G(y)$ is symmetric for all $y$. Then $G$ is $\ell$-sos on $S^{d-1}$, if and only if, $\|y\|^{2(\ell-n)} G(y)$ is a sum of squares.
\end{lemma}

\subsection{Convergence rate of the DPS hierarchy}

In \cite[Theorem 3]{Navascues2009}, Navascues, Owari \& Plenio's proved the following result on the convergence of the sequence of relaxations $(\DPS_{\ell})$ to $\Sep$.

\begin{theorem}[NOP09]\label{NOP result statement}
    For any quantum state $\rho_{AB} \in \DPS_{\ell}$ with reduced state $\rho_A := \tr_B[\rho_{AB}]$, we have
    \begin{align}\label{NOP statement eq}
         (1-t) \rho_{AB} + t  \rho_A \ox \frac{\1_B}{d_B} \;\; \in \;\; \Sep(d)
        \end{align}
        where $t = O\left(\frac{d_B^2}{\ell^2}\right)$, $d_B = \dim(\cH_B)$, and $I_B$ is the identity matrix of dimension $d_B$.
\end{theorem}

Note that the state $\rho_A \otimes \1_B/d_B$ is clearly separable. In words, the result above says that if $\rho_{AB}$ in $\DPS_{\ell}$, then by moving $\rho_{AB}$ in the direction $\rho_A \otimes \1_B / d_B$ by $t=O(d_B^2/\ell^2)$ results in a separable state. In terms of the \emph{Best Separable State} problem, the result of \cite{Navascues2009} has the following immediate implication. We show below how we can recover this result using our Theorem \ref{thm:maintxt} from the previous section.

\begin{theorem}
\label{thm:bssquad2}
Let $M \in \herm(d_A d_B)$ and assume that $(x\ox y)^{\dagger} M (x\ox y) \geq 0$ for all $(x,y) \in \CC^{d_A} \times \CC^{d_B}$. Then
\[
h_{\Sep}(M) \leq h_{\DPS_{\ell}}(M) \leq (1 + \C d_B^2 / \ell^2) h_{\Sep}(M)
\]
for any $\ell \geq \C' d_B$, where $\C,\C' > 0$ is some absolute constant.
\end{theorem}
\begin{proof}
We know from \eqref{eq:hDPSell2} that 
\[
h_{\DPS_{\ell}}(M) = \min \gamma \text{ s.t. } \gamma I - \tilde{\PM}_M(\tilde{y}) \text{ is $\ell$-sos on $\tilde{y} \in S^{2d_B-1}$}.
\]
By assumption we have $0 \leq \tilde{\PM}_M(\tilde{y}) \leq h_{\Sep}(M) I$ for all $\tilde{y} \in S^{2d_B-1}$. Our Theorem \ref{thm:maintxt} from previous section tells us that for $\ell \geq \C d_B$, $\C' d_B^2/\ell^2 + \frac{h_{\Sep}(M) I - \tilde{\PM}_M}{h_{\Sep}(M)}$ is $\ell$-sos on $S^{2d_B-1}$. This implies that $h_{\DPS_{\ell}}(M) \leq h_{\Sep}(M)(1+\C' d_B^2 / \ell^2)$ which is what we wanted.
\end{proof}

\section{Conclusions}

We have shown a quadratic improvement on the convergence rate of the SOS hierarchy on the sphere compared to the previous analysis of Reznick \cite{Reznick1995} and Doherty \& Wehner \cite{Doherty2012}. The proof technique also works for matrix-valued polynomials on the sphere and surprisingly, the bounds we get are independent of the dimension of the matrix polynomial. In the special case of quadratic matrix polynomials, we recover the same rate obtained by Navascues, Owari \& Plenio \cite{Navascues2009} using arguments from quantum information theory.

\newpage

\appendix

\section{Some technical results on polynomials on sphere} \label{proof 1}

We use the following lemma which appears in \cite{Reznick1995}. Recall that the Laplacian of a twice differentiable function $f:\RR^d \rightarrow \RR$ is $\Delta f = \sum_{i=1}^n \frac{\partial^2 f}{\partial x_i^2}$.

\begin{lemma}[\cite{Reznick1995}]\label{Reznick lemma}
    If $f$ is homogeneous polynomial of degree $n$ on $\mathbb R^d$ and $\|f\|_{\infty} \leq M$, then 
    $\|\Delta^k f\|_{\infty} \leq d^k (n)_{2k} M$,
    where $\Delta$ is the Laplace operator and $(n)_m := n(n-1)\cdots (n-(m-1))$ is the falling factorial.
\end{lemma}

\noindent \textbf{Proposition~\ref{prop:boundinfproj} (restatement)\ }
{\it
For any homogeneous polynomial $f$ with degree $2n$, denote its spherical harmonics decomposition as 
    $f(x) = \sum_{k=0}^{n} f_{2k}(x)$ with $f_{j} \in \cH_j^d$. Then for any $k$, it holds
        $\|f_{2k}\|_{\infty} \leq \B_{2n} \|f\|_{\infty}$,
    where $\B_{2n}$ is a constant that depends only on $n$ (and independent of $d$). Also $\B_2 \leq 2$ and $\B_4 \leq 9$.}
\vspace{0.2cm}

\begin{proof}
For simplicity of exposition we prove first the cases $2n=2$ and $2n=4$, before considering the general case.
The result is immediate when $2n=2$ with $\B_{2} = 2$ since the harmonic decomposition of a quadratic polynomial is $f = f_0 + f_2$ with $f_0 = \int_{S^{d-1}} f d\sigma$. Then $|f_0| \leq \|f\|_{\infty}$ and $\|f_2\|_{\infty} = \|f - f_0\|_{\infty} \leq 2 \|f\|_{\infty}$.

The first nontrivial case is $2n=4$. The decomposition of a quartic polynomial on the sphere is $f = f_0 + f_2 + f_4$. Clearly $\|f_0\|_{\infty} \leq \|f\|_{\infty}$. We thus focus on bounding $\|f_2\|_{\infty}$. Since $f$ is homogeneous note that $f(x)$ can be written as $f(x) = \|x\|^4 f_0 + \|x\|^2 f_2(x) + f_4(x)$ for all $x \in \RR^d$. Using well-known identities concerning Laplacian one can check that $\Delta ( f(x) ) = 4(d+2) \|x\|^2 f_0 + 2(d+2) f_2(x)$. (We used that $\Delta f_{2k} = 0$ since the $f_{2k}$ are harmonic, that $\Delta( \|x\|^{2k} ) = 2k(2k+d-2)$ and the identity $\langle x , \nabla g(x) \rangle = 2k g(x)$ for any homogeneous polynomial $g$ of degree $2k$ and $x \in S^{d-1}$.) It thus follows that $f_2(x) = \frac{1}{2d+4}  \Delta f(x) - 2 \|x\|^2 f_0$. By Lemma \ref{Reznick lemma} we know that $\|\Delta f\|_{\infty} \leq 12 d \|f\|_{\infty}$. It thus follows that $\|f_2\|_{\infty} \leq \frac{12d}{2d+4} \|f\|_{\infty} + 2 \|f\|_{\infty} \leq 7 \|f\|_{\infty}$. Finally $\|f_4\|_{\infty} = \|f - (f_0 + f_2)\|_{\infty} \leq 9\|f\|_{\infty}$ by the triangle inequality. Thus $\B_4 \leq 9$.

We now proceed to prove the general case. Let $f(x)$ be a homogeneous polynomial of degree $2n$ and let $f = f_0 + f_2 + \dots + f_{2n}$ be its spherical harmonic decomposition. Note that $f(x)$ has the following expression $f(x) = \sum_{k=0}^{n} \|x\|^{2(n-k)} f_{2k}(x)$ for all $x \in \RR^d$. Since $f_{2k}$ is a spherical harmonic, direct calculations give us 
\begin{equation*}
    \Delta^m \big(\|x\|^{2(n-k)} f_{2k}(x)\big) =
    \begin{cases}
        r_{n,d,m,k} \|x\|^{2(n-k-m)} f_{2k}(x) & \text{ if  $m \leq n-k$}\\
        0 & \text{ otherwise,}
        \end{cases}
\end{equation*}
where,
\begin{align}
    r_{n,d,m,k} = 4^m (n-k)_m \left(n+k+d/2-1\right)_m = O(d^m).
\end{align} 
Thus by linearity, we have
\begin{align}
    \Delta^m f(x) & = \sum_{k=0}^{n-m} \Delta^m \big(\|x\|^{2(n-k)} f_{2k}(x)\big)  = \sum_{k=0}^{n-m} r_{n,d,m,k} \|x\|^{2(n-k-m)} f_{2k}(x).
\end{align}
When restricted on the unit sphere, we have the inequalities
\begin{align}
    \|f_{2(n-m)}\|_{\infty}  \leq \frac{\|\Delta^m f\|_\infty + \sum\limits_{k=0}^{n-m-1} r_{n,d,m,k} \|f_{2k}\|_{\infty}}{r_{n,d,m,n-m}} \leq \frac{d^m (2n)_{2m} \|f\|_\infty + \sum\limits_{k=0}^{n-m-1} r_{n,d,m,k} \|f_{2k}\|_{\infty}}{r_{n,d,m,n-m}},
\end{align}
where the second inequality follows from Lemma~\ref{Reznick lemma}. For any $0 \leq k \leq n-m-1$, we have
\begin{align}
	\frac{r_{n,d,m,k}}{r_{n,d,m,n-m}} = \frac{4^m (n-k)_m \left(n+k+{d}/{2}-1\right)_m}{4^m (m)_m \left(n+(n-m)+{d}/{2}-1\right)_m} \leq \frac{(n-k)_m}{(m)_m} \leq n! \leq (2n)!.
\end{align}
Moreover, for any $m \leq n$ we also have
\begin{align}
	\frac{d^m (2n)_{2m}}{r_{n,d,m,n-m}} = \frac{(d/2)^m (2n)_{2m}}{2^m (m)_m \left(n+(n-m)+{d}/{2}-1\right)_m} \leq \frac{(2n)_{2m}}{2^m(m)_m} \leq (2n)!,
\end{align}
where the first inequality holds since each term in the falling factorial $\left(n+(n-m)+{d}/{2}-1\right)_m$ is no smaller than $d/2$.
Thus we have
\begin{align}
	\|f_{2(n-m)}\|_{\infty} \leq (2n)! \left[\|f\|_\infty+\sum\limits_{k=0}^{n-m-1} \|f_{2k}\|_{\infty}\right]. 
\end{align}
By induction, we have the estimation
\begin{align}
	\|f_{2k}\|_{\infty} \leq \|f\|_\infty (2n)! \left[1+(2n)!\right]^{k} \leq \|f\|_\infty (2n)! \left[1+(2n)!\right]^{n}, \quad \forall k.
\end{align}
Thus we have $\B_{2n} \leq (2n)! \left[1+(2n)!\right]^{n}$ independent of $d$.

\end{proof}

We can extend Proposition \ref{prop:boundinfproj} to matrix-valued polynomials on the sphere.

\begin{proposition}
\label{prop:boundinfprojmatrix}
Let $F(x) \in \mathbf{S}^k[x]$ be a $k\times k$ symmetric matrix-valued polynomial of degree $2n$. Let $\|F\|_{\infty} = \max_{x \in S^{d-1}} \|F(x)\|$ where $\|\cdot \|$ denotes the spectral norm. If $F = F_0 + F_2 + \dots +F_{2n}$ is the harmonic decomposition of $F$, then $\|F_{2k}\|_{\infty} \leq \B_{2n} \|F\|_{\infty}$ where $\B_{2n}$ is the constant from Proposition \ref{prop:boundinfproj}.
\end{proposition}
\begin{proof}
We can assume without loss of generality that $\|F\| = 1$. Note that $\|F\| = \max_{x \in S^{d-1}} \max_{y \in S^{k-1}} |y^\sfT F(x) y|$.
For any fixed $y \in S^{k-1}$ define the real-valued polynomial $f_y(x) = y^\sfT F(x) y$. By assumption on $F$ we know that $\|f_y\|_{\infty} = \max_{x \in S^{d-1}} |y^\sfT F(x) y| \leq 1$. The spherical harmonic decomposition of $f_y$ is given by $f_y(x) = \sum_{k=0}^{2n} y^\sfT F_{2k}(x) y$ since, for fixed $y$, $y^\sfT F_{2k}(x) y$ is a linear combination of the entries of $F_{2k}$ which are all in $\cH_{2k}^d$. It thus follows from Proposition \ref{prop:boundinfproj} that $\|y^\sfT F_{2k}(\cdot) y\|_{\infty} \leq \B_{2n}$. This is true for all $y \in S^{k-1}$ thus we get $\max_{y \in S^{k-1}} \max_{x \in S^{d-1}} |y^\sfT F_{2k}(x) y| \leq \B_{2n}$, i.e., $\|F_{2k}\| \leq \B_{2n}$ as desired.
\end{proof}

\vspace{0.2cm}
We will also need the following technical result about Gegenbauer polynomials.
\begin{proposition}
\label{prop:Cilb}
Let $C_i(t)$ be the Gengebauer polynomial of degree $i$. Then for the curve of $C_i$ lies above its tangent at $t=1$, i.e., $C_i(t) \geq C_i'(1)(t-1) + C_i(1)$ for all $t \in [-1,1]$.
\end{proposition}
\begin{proof}
Let $l(t) = C_i'(1)(t-1) + C_i(1)$. Let $\alpha = \max \left\{ x \in (0,1) : C_i'(x) = 0\right\}$. It is known that $C_i(\alpha) < 0$ and that $|C_i(t)| \leq |C_i(\alpha)|$ for all $t \in [0,\alpha]$ (see \cite[18.14.16]{NIST:DLMF}).
By standard arguments on orthogonal polynomials, we know that $C_i'' \geq 0$ on $[\alpha,\infty)$. Thus the inequality $C_i(t) \geq l(t)$ is true on $t \in [\alpha,1]$. For $t \in [0,\alpha]$ it also has to be true since
\[
C_i(t) \geq -|C_i(\alpha)| = C_i(\alpha) \geq l(\alpha) \geq l(t).
\]
Using the fact that $C_i'(1)/C_i(1) = i(i+d-2)/(d-1)$, one can easily check that $l(0) \leq -C_i(1)$, and so $C_i(t) \geq -C_i(1) \geq l(t)$ for all $t \in [-1,0]$.
\end{proof}

\section{Duality relations DPS and SOS (Theorem \ref{thm:duality})}\label{Duality relation}

In this section we prove that for any integer $\ell\geq 1$, we have the duality relation
\begin{align}
\label{eq:DPSkstar}
    \DPScone_\ell^* = \left\{M : \|y\|^{2(\ell-1)} p_M \text{ is rsos}\right\}.
\end{align}

The key is the following lemma which gives a semidefinite programming characterization of the right-hand side of \eqref{eq:DPSkstar}.

\begin{lemma}\label{k local SOS to W lemma}
    For any $M_{AB} \in \herm(\cH_A \otimes \cH_B)$ and integer $\ell\geq 1$, then $\|y\|^{2(\ell-1)} p_M$ is a rsos if and only if there exist positive semidefinite operators $W_{s,AB_{[\ell]}} \geq 0$, $s = 0,1,\cdots,\ell$ such that
    \begin{align}
        \|y\|^{2(\ell-1)} p_M & = \sum_{s=0}^\ell \left(x \ox \bar y^{\ox s} \ox y^{\ox \ell-s}\right)^\dagger W_{s,AB_{[\ell]}} \left(x\ox \bar y^{\ox s} \ox y^{\ox \ell-s}\right)\\
        & = \sum_{s=0}^\ell \left(x \ox y^{\ox \ell}\right)^\dagger W_{s,AB_{[\ell]}}^{\sfT_{B_{[s]}}} \left(x \ox y^{\ox \ell}\right).\label{k local sos to W lemma tmp1}
    \end{align}
\end{lemma}

The proof of the previous lemma is based on analyzing the biquadratic structure of $p_M$ to see which monomials can appear in a sum-of-squares decomposition of $\|y\|^{2(\ell-1)} p_M$. The proof is deferred to the end of this section.

Using Lemma \ref{k local SOS to W lemma}, the proof of \eqref{eq:DPSkstar} follows from standard duality arguments which we know explain.

First, we can dualize the semidefinite programming definition of $\DPS_\ell$ to get
\[
\begin{aligned}
\DPScone_\ell^* = \Biggl\{ M_{AB_1} : & \;\; M_{AB_1}\ox \1_{B_{[2:\ell]}} = \big(Y_{AB_{[\ell]}} - \Pi_\ell Y_{AB_{[\ell]}} \Pi_\ell \big) + \sum_{s=0}^\ell W_{s,AB_{[\ell]}}^{\sfT_{B_{[s]}}} \\
& \quad \quad \text{ where } Y_{AB_{[\ell]}} \in \herm,\, W_{s,AB_{[\ell]}} \geq 0, \forall s \in [0:\ell] \Biggr\}.
\end{aligned}
\]
The variable $W_{s,AB_{[\ell]}}$ for $s=0$ (resp. $s=1,\ldots,\ell$) is the dual variable for the positivity constraint on $\rho_{AB_{[\ell]}}$ (resp. PPT constraint \eqref{eq:extensionpptcond}).

\begin{proof}[Proof of Theorem \ref{thm:duality}]
    The proof consists of two directions. Assume $M \in \DPScone_\ell^*$. Then there exists a Hermitian operator $Y_{AB_{[\ell]}}$, and positive semidefinite operators $W_{s,AB_{[\ell]}} \geq 0$, $s = 0,1,\cdots,\ell$ such that
    \begin{align}\label{duality theorem eq tmp2}
        M_{AB_1}\ox \1_{B_{[2:\ell]}} = \left[Y_{AB_{[\ell]}} - (I\ox \Pi_\ell)Y_{AB_{[\ell]}} (I\ox \Pi_\ell)\right] + \sum_{s=0}^\ell W_{s,AB_{[\ell]}}^{\sfT_{B_{[s]}}}.
    \end{align}
    Recalling that $\Pi_\ell$ is the projector onto the symmetric subspace, we have $\Pi y^{\ox \ell} = y^{\ox \ell}$ for any vector $y$. Thus 
    \begin{align}
        \left(x\ox y^{\ox \ell}\right)^\dagger \left(Y_{AB_{[\ell]}} - (I \otimes \Pi_\ell) Y_{AB_{[\ell]}} (I \otimes \Pi_\ell)\right)\left(x\ox y^{\ox \ell}\right) = 0,\quad \forall\, x \in \cH_A, y \in \cH_B.
    \end{align} 
    Evaluating Eq.~\eqref{duality theorem eq tmp2} on both sides at the state $x\ox y^{\ox \ell}$, we have
    \begin{align}\label{duality theorem eq tmp1}
        \|y\|^{2(\ell-1)} p_M = \left(x\ox y^{\ox \ell}\right)^\dagger M_{AB_1}\ox \1_{B_{[2:\ell]}}\left(x\ox y^{\ox \ell}\right) = \sum_{s=0}^\ell \left(x \ox y^{\ox \ell}\right)^\dagger W_{s,AB_{[\ell]}}^{\sfT_{B_{[s]}}}\left(x\ox y^{\ox \ell}\right).
    \end{align}
    According to Proposition~\ref{k local SOS to W lemma}, we have $\|y\|^{2(\ell-1)} p_M$ is a rsos.

    On the other hand, suppose $\|y\|^{2(\ell-1)} p_M$ is a rsos. From Proposition~\ref{k local SOS to W lemma}, there exists positive semidefinite operators $W_{s,AB_{[\ell]}} \geq 0$, $s = 0,1,\cdots,\ell$ such that Eq.~\eqref{duality theorem eq tmp1} holds.
    Since $y^{\ox \ell}$ forms a basis on the symmetric subspace of $\cH_{B_1}\ox \cH_{B_2}\ox \cdots \ox \cH_{B_\ell}$, it implies that the operators 
    \begin{align}
    M_{AB_1}\ox \1_{B_{[2:\ell]}}\quad \text{and} \quad \sum_{s=0}^\ell W_{s,AB_{[\ell]}}^{\sfT_{B_{[s]}}}
    \end{align}
     coincide when restricted on the symmetric subspace $\Sym(\cH^{\otimes \ell})$. That is,
    \begin{align}
        (I\ox \Pi_\ell) (M_{AB_1}\ox \1_{B_{[2:\ell]}}) (I\ox \Pi_\ell) = (I\ox \Pi_\ell) \left(\sum_{s=0}^\ell W_{s,AB_{[\ell]}}^{\sfT_{B_{[s]}}}\right) (I\ox \Pi_\ell).
    \end{align}
    Take the Hermitian operator
    \begin{align}
        Y_{AB_{[\ell]}}:= M_{AB_1}\ox \1_{B_{[2:\ell]}} - \sum_{s=0}^\ell W_{s,AB_{[\ell]}}^{\sfT_{B_{[s]}}}.
    \end{align}
    Then by the definition of $Y_{AB_{[\ell]}}$, we have $\Pi_\ell Y_{AB_{[\ell]}} \Pi_\ell = 0$ and 
    \begin{align}
        M_{AB_1}\ox \1_{B_{[2:\ell]}} = \left[Y_{AB_{[\ell]}} - (I\ox \Pi_\ell)Y_{AB_{[\ell]}} (I\ox \Pi_\ell)\right] + \sum_{s=0}^\ell W_{s,AB_{[\ell]}}^{\sfT_{B_{[s]}}},
    \end{align}
    which implies $M_{AB_1} \in \DPScone_\ell^*$. 
\end{proof}

It remains to prove Lemma \ref{k local SOS to W lemma}. To have an easier understanding of the result in Lemma~\ref{k local SOS to W lemma}, let us first have a look at the special case on the second level of the hierarchy, i.e, $\ell = 2$. This will give us the key idea without loss of generality, and the higher level case is just a straightforward generalization. 

\vspace{0.2cm}
\noindent{\textbf{Lemma~\ref{k local SOS to W lemma} [special case $\ell=2$]} {\it
    For any Hermitian operator $M_{AB_1}$, we have that $\|y\|^2 p_M$ is rsos if and only if there exist positive semidefinite operators $W_{0,AB_1B_2}, W_{1,AB_1B_2}, W_{2,AB_1B_2} \geq 0$, such that
    \begin{equation}\label{duality second tmp1}
    \begin{aligned}
        \|y\|^2 p_M(x,\bar x, y,\bar y) &= (x \ox y \ox y)^{\dagger} W_{0,AB_1B_2} (x \ox y \ox y)\\
         &\quad + (x \ox \bar{y} \ox y)^{\dagger} W_{1,AB_1B_2} (x \ox \bar{y} \ox y)\\
         &\quad + (x \ox \bar{y} \ox \bar{y})^{\dagger} W_{2,AB_1B_2} (x \ox \bar{y} \ox \bar{y}).
    \end{aligned}
\end{equation}}

\begin{proof}
    If there exist operators $W_0,W_1,W_2\geq 0$ such that Eq.~\eqref{duality second tmp1} holds, then $\|y\|^2 p_M$ can be shown to be rsos by using the spectral decompostion of $W_i$. For the converse suppose $\|y\|^2 p_M$ is rsos. Then there exist polynomials $f_m(x,\bar x, y,\bar y)$ such that $\|y\|^2 p_M(x,\bar x, y, \bar y) = \sum_m f_m(x,\bar x,y, \bar y)^2$.
    Since the monomials of $\|y\|^2 p_M(x,\bar x, y, \bar y)$ are all of the forms $\bar x_i x_m \bar y_j  y_k y_r \bar y_r$ (they are degree 2 in $(x,\bar x)$ and degree 4 in $(y,\bar y)$, then the possible monomials of $f_m(x,\bar x, y, \bar y)$ can only be given by
    \begin{align}
     \big\{x_iy_jy_k, x_i \bar y_j y_k, x_i \bar y_j \bar y_k, \bar x_iy_jy_k, \bar x_i y_j \bar y_k, \bar x_i \bar y_j \bar y_k\big\}.
    \end{align}
    The existence of any other monomials in $f_m(x,\bar x, y, \bar y)$, such as $x_ix_jy_k$, will not be compatible with the monomials in $\|y\|^2 p_M(x,\bar x, y,\bar{y})$.
    Thus the most general form of $f_m(x,\bar x, y, \bar y)$ can be written as the linear combinations,
    \begin{align}
        f_m(x,\bar x, y, \bar y)  = & \sum_{i,j,k} a_{i,j,k}^{m,0} x_i y_j y_k + \sum_{i,j,k} a_{i,j,k}^{m,1} x_i \bar y_j y_k + \sum_{i,j,k} a_{i,j,k}^{m,2} x_i \bar y_j \bar y_k \notag\\
         + & \sum_{i,j,k} b_{i,j,k}^{m,0} \bar x_i \bar y_j \bar y_k + \sum_{i,j,k} b_{i,j,k}^{m,1} \bar x_i  y_j \bar y_k + \sum_{i,j,k} b_{i,j,k}^{m,2} \bar x_i  y_j  y_k.
    \end{align}
    Since $f_m(x,\bar x, y, \bar y) \in \mathbb R$, we have 
    \begin{align}
        \bar a_{i,j,k}^{m,0} = b_{i,j,k}^{m,0},\ 
        \bar a_{i,j,k}^{m,1} = b_{i,j,k}^{m,1},\
        \bar a_{i,j,k}^{m,2} = b_{i,j,k}^{m,2}, \quad \forall\ i,j,k,m.
    \end{align}
    Comparing the monomials of $\sum_m f_m(x,\bar x, y, \bar y)^2$ and $\bar x_i x_m \bar y_j  y_k y_r \bar y_r$, the terms, such as 
    \begin{align}
        \sum_m \left(\sum_{i,j,k} a_{i,j,k}^{m,0} x_i y_j y_k\right)\left(\sum_{i,j,k} a_{i,j,k}^{m,1} x_i \bar y_j y_k\right)
    \end{align} 
    have to vanish, since the resulting monomial $x_i y_j y_k x_{i'} \bar y_{j'} y_{k'}$ is not compatible with $\bar x_i x_m \bar y_j  y_k y_r \bar y_r$.
    After we get rid of those incompatible monomials, we have
    \begin{align}
        \sum_m f_m(x,\bar x, y, \bar y)^2 = 2 \sum_m \Bigg(\Big|\sum_{i,j,k} a_{i,j,k}^{m,0} x_i y_j y_k\Big|^2 + \Big|\sum_{i,j,k} a_{i,j,k}^{m,1} x_i \bar y_j y_k\Big|^2 + \Big|\sum_{i,j,k} a_{i,j,k}^{m,2} x_i \bar y_j \bar y_k\Big|^2\Bigg).
    \end{align}
    Then we can construct matrices $W_0,W_1,W_2$ whose elements are respectively given by
    \begin{align}
        (W_0)_{i,j,k;r,s,t} = 2 \sum_m \bar a_{i,j,k}^{m,0}\, a_{r,s,t}^{m,0}, \ (W_1)_{i,j,k;r,s,t} = 2 \sum_s \bar a_{i,j,k}^{m,1}\, a_{r,s,t}^{m,1}, \ (W_2)_{i,j,k;r,s,t} = 2 \sum_s \bar a_{i,j,k}^{m,2}\, a_{r,s,t}^{m,2}.
    \end{align}
    By construction, we know that $W_0,W_1,W_2 \geq 0$ and 
    \begin{equation}
    \begin{aligned}
\|y\|^2 p_M(x,\bar x, y, \bar y) = \textstyle\sum_m f_m(x,\bar x, y, \bar y)^2 &= (x \ox y \ox y)^{\dagger} W_{0,AB_1B_2} (x \ox y \ox y)\\
         &\quad + (x \ox \bar{y} \ox y)^{\dagger} W_{1,AB_1B_2} (x \ox \bar{y} \ox y)\\
         &\quad + (x \ox \bar{y} \ox \bar{y})^{\dagger} W_{2,AB_1B_2} (x \ox \bar{y} \ox \bar{y}),
    \end{aligned}
    \end{equation}
    which completes the proof.
\end{proof}

\vspace{0.2cm}
\noindent{\textbf{Lemma~\ref{k local SOS to W lemma} [general result, restatement]}}
{\it For any Hermitian operator $M_{AB_1}$ and integer $\ell \geq 1$, we have that $\|y\|^{2(\ell-1)} p_M$ is rsos if and only if there exist positive semidefinite operators $W_{s,AB_{[\ell]}} \geq 0$, $s \in [0:\ell]$ such that
\begin{align}\label{proof k local sos to W lemma tmp1}
        \|y\|^{2(\ell-1)} p_M & = \sum_{s=0}^\ell \left(x \ox \bar y^{\ox s} \ox y^{\ox \ell-s}\right)^\dagger W_{s,AB_{[\ell]}} \left(x\ox \bar y^{\ox s} \ox y^{\ox \ell-s}\right)\\
        & = \sum_{s=0}^\ell \left(x \ox y^{\ox \ell}\right)^\dagger W_{s,AB_{[\ell]}}^{\sfT_{B_{[s]}}} \left(x \ox y^{\ox \ell}\right).\
    \end{align}}

\begin{proof}
    Note that the second equality trivially holds due to the equation $x^\dagger Zx = (\bar x)^\dagger Z^\sfT (\bar x)$. We will prove the first equality.
    If Eq.~\eqref{proof k local sos to W lemma tmp1} holds for positive semidefinite operators $W_{s,AB_{[\ell]}}$, then it is easy to check that $\|y\|^{2(\ell-1)} p_M(x,\bar x,y,\bar y)$ is a rsos by using the spectral decomposition of $W_{s,AB_{[\ell]}}$. 

    On the other hand, if $\|y\|^{2(\ell-1)} p_M(x,\bar x,y,\bar y)$ is a rsos, by definition there exist Hermitian polynomials $f_m(x,\bar x, y, \bar y)$ such that $\|y\|^{2(\ell-1)} p_M(x,\bar x,y,\bar y) = \sum_m f_m(x,\bar x, y, \bar y)^2$. In the following, we will compare the monomials on both sides of this equation and explicitly construct $W_{s,AB_{[\ell]}}$ from the coefficients of $\sum_m f_m(x,\bar x, y, \bar y)^2$.
    We first note that the monomials of $\|y\|^{2(\ell-1)} p_M(x,\bar x,y,\bar y)$ are all of the form 
    \begin{align}\label{local sos monomials form}
    x_{t} \bar x_{t'} \prod_{i=1}^\ell y_{r_i} \bar y_{r_i'},
    \end{align}
which is of degree $2$ and $2\ell$ with respect to $x$ and $y$, respectively.
Then the possible monomials of $f_m(x,\bar x, y, \bar y)$ can only be of degree $1$ and $\ell$ with respect to $x$ and $y$, respectively. That is, the possible monomials are given by
\begin{align}
 \left\{x_t \prod_{i=1}^s \bar y_{r_i} \prod_{i=s+1}^\ell y_{r_i}\right\}_{s=0}^\ell  \quad \text{and} \quad  \left\{\bar x_{t} \prod_{i=1}^s y_{r_i} \prod_{i=s+1}^\ell \bar y_{r_i}\right\}_{s=0}^\ell,
\end{align}
where we denote the term $\prod_{i=s_1}^{s_2} (\cdot) = 1$ if $s_2 < s_1$. These monomials are basically formed by the ones with different number of complex conjugation over the symbol $y$.
Therefore, the most general form of $f_m(x,\bar x, y, \bar y)$ can be written as a linear combination of these monomials:
\begin{align}
    f_m(x,\bar x, y, \bar y)  = \sum_{s=0}^\ell \sum_{t, r_{[\ell]}} a^{m,s}_{t,r_{[\ell]}} \left(x_t \prod_{i=1}^s \bar y_{r_i} \prod_{i=s+1}^\ell y_{r_i}\right) +  \sum_{s=0}^\ell \sum_{t, r_{[\ell]}} b^{m,s}_{t,r_{[\ell]}} \left(\bar x_{t} \prod_{i=1}^s y_{r_i} \prod_{i=s+1}^\ell \bar y_{r_i}\right).
\end{align}
Since $f_m(x,\bar x, y, \bar y) \in \mathbb R$ for all $x,y$, we know that the coefficients between the conjugate monomials have to be conjugate with each other. That is, $\bar a^{m,s}_{t,r_{[\ell]}} = b^{m,s}_{t,r_{[\ell]}}$ holds for all $m,s,t,r_{[\ell]}$.
Comparing the monomials of $\sum_m f_m(x,\bar x, y, \bar y)^2$ and the monomials in Eq.~\eqref{local sos monomials form}, we have
\begin{align}
\|y\|^{2(\ell-1)} p_M(x,\bar x,y,\bar y) = \sum_m f_m(x,\bar x, y, \bar y)^2 = 2 \sum_m \sum_{s=0}^\ell \Bigg| \sum_{t, r_{[\ell]}} a^{m,s}_{t,r_{[\ell]}} \left(x_t \prod_{i=1}^s \bar y_{r_i} \prod_{i=s+1}^\ell y_{r_i}\right)\Bigg|^2.
\end{align}
For any $s \in [0:\ell]$, we construct the matrix $W_s$ whose elements are given by
\begin{align}
(W_{s})_{t',r'_{[\ell]};\, t,r_{[\ell]}} := 2 \sum_m \bar a^{m,s}_{t',r'_{[\ell]}} a^{m,s}_{t,r_{[\ell]}}.
\end{align}
Then we have that $W_s \geq 0$, $\forall s\in [0:\ell]$ and 
\begin{align}
\|y\|^{2(\ell-1)} p_M(x,\bar x,y,\bar y) = \sum_m f_m(x,\bar x, y, \bar y)^2 = \sum_{s=0}^\ell \left(x \ox \bar y^{\ox s} \ox y^{\ox \ell-s}\right)^\dagger W_{s} \left(x\ox \bar y^{\ox s} \ox y^{\ox \ell-s}\right),
\end{align}
which completes the proof.
\end{proof}

The above argument also works for csos polynomials with slight modifications.  

\begin{lemma}\label{csos equivalence lemma}
For any Hermitian operator $M_{AB_1}$ and integer $\ell \geq 1$, we have that $\|y\|^{2(\ell-1)} p_M$ is csos if and only if there exists a positive semidefinite operator $W_{AB_{[\ell]}} \geq 0$, such that
    \begin{align}\label{proof k local sos to W lemma tmp2}
        \|y\|^{2(\ell-1)} p_M(x,\bar x, y, \bar y) = \left(x \ox y^{\ox \ell}\right)^\dagger W_{AB_{[\ell]}}\left(x \ox y^{\ox \ell}\right).
\end{align}
\end{lemma}
\begin{proof}
	If there exists $W_{AB_{[\ell]}} \geq 0$ such that Eq.~\eqref{proof k local sos to W lemma tmp2} holds, we can check that $\|y\|^{2(\ell-1)} p_M(x,\bar x,y,\bar y)$ is a csos by using the spectral decomposition of $W$.
On the other hand, if $\|y\|^{2(\ell-1)} p_M(x,\bar x,y,\bar y)$ is a csos, by definition there exist polynomials $f_m(x, y)$ such that $\|y\|^{2(\ell-1)} p_M(x,\bar x,y,\bar y) = \sum_m |f_m(x, y)|^2$. In the following, we will compare the monomials on both sides of this equation and explicitly construct $W_{AB_{[\ell]}}$ from the coefficients of $\sum_m |f_m(x, y)|^2$.
We first note that the monomials of $\|y\|^{2(\ell-1)} p_M(x,\bar x,y,\bar y)$ are all of the form $x_{t} \bar x_{t'} \prod_{i=1}^\ell y_{r_i} \bar y_{r_i'}$,
which is of degree $2$ and $2\ell$ with respect to $x$ and $y$, respectively.
Then the possible monomials of $f_m(x,y)$ can only be of degree $1$ and $\ell$ with respect to $x$ and $y$, respectively. Furthermore, by definition $f_m(x,y)$ are polynomials with respect to $x,y$ alone, thus the only possible monomial of $f_m(x,y)$ is
$x_t \prod_{i=1}^\ell y_{r_i}$ and we have the general form of 
 $f_m(x,y)$ as $f_m(x,y)  = \sum_{t, r_{[\ell]}} a^{m}_{t,r_{[\ell]}} \big(x_t \prod_{i=1}^\ell y_{r_i}\big)$ with coefficients $a^{m}_{t,r_{[\ell]}}$. Define the matrix $W$ with elements 
\begin{align}
W_{t',r'_{[\ell]};\, t,r_{[\ell]}} := \sum_m \bar a^{m}_{t',r'_{[\ell]}} a^{m}_{t,r_{[\ell]}}.
\end{align}
Then we have $W \geq 0$, and 
\begin{align}
\|y\|^{2(\ell-1)} p_M(x,\bar x,y,\bar y) = \sum_m |f_m(x, y)|^2 = \left(x \ox y^{\ox \ell}\right)^\dagger W_{AB_{[\ell]}}\left(x \ox y^{\ox \ell}\right),
\end{align}
which completes the proof.
\end{proof}

Finally the result of
$\EXTcone_\ell^* = \left\{M : \|y\|^{2(\ell-1)} p_M \text{ is csos}\right\}$ can be proved in a similar way by using Lemma~\ref{csos equivalence lemma}.


\bibliographystyle{alpha_abbrv}
\bibliography{bib}

\newcommand{\etalchar}[1]{$^{#1}$}
\begin{thebibliography}{CKMR07}

\bibitem[ADGR04]{area2004zeros}
I.~Area, D.~Dimitrov, E.~Godoy, and A.~Ronveaux.
\newblock Zeros of {G}egenbauer and {H}ermite polynomials and connection
  coefficients.
\newblock {\em Mathematics of Computation}, 73(248):1937--1951, 2004.

\bibitem[Bax71]{baxley1971extreme}
J.~V. Baxley.
\newblock Extreme eigenvalues of {T}oeplitz matrices associated with certain
  orthogonal polynomials.
\newblock {\em SIAM Journal on Mathematical Analysis}, 2(3):470--482, 1971.

\bibitem[BBH{\etalchar{+}}12]{barak2012hypercontractivity}
B.~Barak, F.~G. Brandao, A.~W. Harrow, J.~Kelner, D.~Steurer, and Y.~Zhou.
\newblock Hypercontractivity, sum-of-squares proofs, and their applications.
\newblock In {\em Proceedings of the forty-fourth annual ACM symposium on
  Theory of computing}, pages 307--326. ACM, 2012.

\bibitem[BGG{\etalchar{+}}17]{bhattiprolu}
V.~Bhattiprolu, M.~Ghosh, V.~Guruswami, E.~Lee, and M.~Tulsiani.
\newblock Weak decoupling, polynomial folds and approximate optimization over
  the sphere.
\newblock In {\em 2017 IEEE 58th Annual Symposium on Foundations of Computer
  Science (FOCS)}, pages 1008--1019. IEEE, 2017.

\bibitem[BKS17]{Barak2017}
B.~Barak, P.~K. Kothari, and D.~Steurer.
\newblock Quantum entanglement, sum of squares, and the log rank conjecture.
\newblock In {\em Proceedings of the 49th Annual ACM SIGACT Symposium on Theory
  of Computing}, pages 975--988. ACM, 2017.

\bibitem[Ble04]{blekherman2004convexity}
G.~Blekherman.
\newblock Convexity properties of the cone of nonnegative polynomials.
\newblock {\em Discrete \& Computational Geometry}, 32(3):345--371, 2004.

\bibitem[Bri61]{brickman1961field}
L.~Brickman.
\newblock On the field of values of a matrix.
\newblock {\em Proceedings of the American Mathematical Society}, 12(1):61--66,
  1961.

\bibitem[CKMR07]{definettioneandhalf}
M.~Christandl, R.~K{\"o}nig, G.~Mitchison, and R.~Renner.
\newblock One-and-a-half quantum de {F}inetti theorems.
\newblock {\em Communications in Mathematical Physics}, 273(2):473--498, 2007.

\bibitem[DJ12]{driver2012bounds}
K.~Driver and K.~Jordaan.
\newblock Bounds for extreme zeros of some classical orthogonal polynomials.
\newblock {\em Journal of Approximation Theory}, 164(9):1200--1204, 2012.

\bibitem[DK08]{deklerksurvey}
E.~De~Klerk.
\newblock The complexity of optimizing over a simplex, hypercube or sphere: a
  short survey.
\newblock {\em Central European Journal of Operations Research},
  16(2):111--125, 2008.

\bibitem[dKL19]{klerklaurentsphere}
E.~de~Klerk and M.~Laurent.
\newblock Convergence analysis of a {L}asserre hierarchy of upper bounds for
  polynomial minimization on the sphere.
\newblock {\em arXiv preprint arXiv:1904.08828}, 2019.

\bibitem[dKLP05]{de2005equivalence}
E.~de~Klerk, M.~Laurent, and P.~Parrilo.
\newblock On the equivalence of algebraic approaches to the minimization of
  forms on the simplex.
\newblock In {\em Positive Polynomials in Control}, pages 121--132. Springer,
  2005.

\bibitem[{\relax DLMF}]{NIST:DLMF}
{\it NIST Digital Library of Mathematical Functions}.
\newblock http://dlmf.nist.gov/, Release 1.0.22 of 2019-03-15.
\newblock F.~W.~J. Olver, A.~B. {Olde Daalhuis}, D.~W. Lozier, B.~I. Schneider,
  R.~F. Boisvert, C.~W. Clark, B.~R. Miller and B.~V. Saunders, eds.

\bibitem[DP09]{putinardangelo}
J.~P. D'Angelo and M.~Putinar.
\newblock Polynomial optimization on odd-dimensional spheres.
\newblock In {\em Emerging applications of algebraic geometry}, pages 1--15.
  Springer, 2009.

\bibitem[DPS02]{Doherty2002}
A.~C. Doherty, P.~A. Parrilo, and F.~M. Spedalieri.
\newblock Distinguishing separable and entangled states.
\newblock {\em Physical Review Letters}, 88(18):187904, 2002.

\bibitem[DPS04]{Doherty}
A.~C. Doherty, P.~A. Parrilo, and F.~M. Spedalieri.
\newblock {Complete family of separability criteria}.
\newblock {\em Physical Review A}, 69(2):022308, 2004.

\bibitem[DW12]{Doherty2012}
A.~C. Doherty and S.~Wehner.
\newblock Convergence of {SDP} hierarchies for polynomial optimization on the
  hypersphere.
\newblock {\em arXiv:1210.5048}, 2012.

\bibitem[Gur03]{gurvits2003classical}
L.~Gurvits.
\newblock Classical deterministic complexity of {E}dmonds' problem and quantum
  entanglement.
\newblock In {\em Proceedings of the thirty-fifth annual ACM symposium on
  Theory of computing}, pages 10--19. ACM, 2003.

\bibitem[HHH01]{Horodecki2001}
M.~Horodecki, P.~Horodecki, and R.~Horodecki.
\newblock Mixed-state entanglement and quantum communication.
\newblock In {\em Quantum information}, pages 151--195. Springer, 2001.

\bibitem[Hsu38]{Hsu1938}
H.-y. Hsu.
\newblock {Certain integrals and infinite series involving ultra-spherical
  polynomials and Bessel functions}.
\newblock {\em Duke Mathematical Journal}, 4(2):374--383, 1938.

\bibitem[KM09]{KMdefinetti}
R.~Koenig and G.~Mitchison.
\newblock A most compendious and facile quantum de {F}inetti theorem.
\newblock {\em Journal of Mathematical Physics}, 50(1):012105, 2009.

\bibitem[Las01]{lasserre2001global}
J.~B. Lasserre.
\newblock Global optimization with polynomials and the problem of moments.
\newblock {\em SIAM Journal on optimization}, 11(3):796--817, 2001.

\bibitem[LBC{\etalchar{+}}00]{Lewenstein2000}
M.~Lewenstein, D.~Bru{\ss}, J.~I. Cirac, B.~Kraus, M.~Ku{\'{s}},
  J.~Samsonowicz, A.~Sanpera, and {\&}.~R. Tarrach.
\newblock {Separability and distillability in composite quantum systems-a
  primer}.
\newblock {\em Journal of Modern Optics}, 47(14):2481--2499, 2000.

\bibitem[Nes03]{nesterov2003random}
Y.~Nesterov.
\newblock Random walk in a simplex and quadratic optimization over convex
  polytopes.
\newblock Technical report, CORE, 2003.

\bibitem[NOP09]{Navascues2009}
M.~Navascues, M.~Owari, and M.~B. Plenio.
\newblock {The power of symmetric extensions for entanglement detection}.
\newblock {\em Physical Review A - Atomic, Molecular, and Optical Physics},
  80(5):1--16, 2009.

\bibitem[Par65]{Parter1965}
S.~V. Parter.
\newblock {Remarks on the extreme eigenvalues of {T}oeplitz forms associated
  with orthogonal polynomials}.
\newblock {\em Journal of Mathematical Analysis and Applications},
  12(3):456--470, 1965.

\bibitem[Par00]{parrilo2000structured}
P.~A. Parrilo.
\newblock {\em Structured semidefinite programs and semialgebraic geometry
  methods in robustness and optimization}.
\newblock PhD thesis, California Institute of Technology, 2000.

\bibitem[Par13]{Parrilo-unpublished}
P.~A. Parrilo.
\newblock Approximation quality of {SOS} relaxations, 2013.
\newblock Talk at ICCOPT 2013.

\bibitem[PT07]{polik2007survey}
I.~P{\'o}lik and T.~Terlaky.
\newblock A survey of the s-lemma.
\newblock {\em SIAM Review}, 49(3):371--418, 2007.

\bibitem[Rez95]{Reznick1995}
B.~Reznick.
\newblock {Uniform denominators in Hilbert's seventeenth problem}.
\newblock {\em Mathematische Zeitschrift}, 220(1):75--97, 1995.

\bibitem[Wor76]{woronowicz1976positive}
S.~L. Woronowicz.
\newblock Positive maps of low dimensional matrix algebras.
\newblock {\em Reports on Mathematical Physics}, 10(2):165--183, 1976.

\end{thebibliography}

\end{document}